\documentclass[11pt,a4paper]{article}
\usepackage{amssymb,latexsym}
\usepackage{bm, amsmath, amssymb, amsthm}

\def\<{\langle}
\def\>{\rangle}
\def\eps{\varepsilon}
\def\RR{\mathbb{R}}

\def\calf{\mathcal{F}}
\def\tr{\operatorname{Tr\,}}
\def\id{\operatorname{id\,}}
\def\Div{\operatorname{div}}

\def\Ric{\operatorname{Ric}}
\def\vol{\operatorname{vol}}
\newcommand{\Sm}{{\mbox{\rm S}}}
\newcommand\const{\operatorname{const}}
\def\eq{\hspace*{-1.5mm}&=&\hspace*{-1.5mm}}
\def\plus{\hspace*{-1.5mm}&+&\hspace*{-1.5mm}}
\def\minus{\hspace*{-1.5mm}&-&\hspace*{-1.5mm}}
\def\dt{\partial_t}

\newtheorem{thm}{Theorem}[section]
\newtheorem{cor}[thm]{Corollary}
\newtheorem{lem}[thm]{Lemma}
\newtheorem{prop}[thm]{Proposition}
\newtheorem{example}[thm]{Example}
\newtheorem{defn}{Definition}[section]
\newtheorem{rem}{Remark}[section]

\title{The Einstein-Hilbert type action \\ on foliated pseudo-Riemannian manifolds}

\author{
       Vladimir Rovenski\footnote{Mathematical Department, University of Haifa, Mount Carmel, Haifa, 31905, Israel
       \newline e-mail: {\tt rovenski@math.haifa.ac.il}}\ \
        \ and \
       Tomasz Zawadzki
       \footnotemark[1]
       \footnote{Katedra Geometrii, Uniwersytet \L\'{o}dzki, ul. Banacha 22, 90-238 \L\'{o}d\'{z}, Poland
        \newline e-mail: {\tt zawadzki@math.uni.lodz.pl}}
       }

\begin{document}

\date{}

\maketitle

\begin{abstract}

We develop variation formulas on almost-product (e.g. foliated) pseudo-Rieman\-nian manifolds,
and we consider variations of metric preserving orthogonality of the distributions.
These formulae are applied to Einstein-Hilbert type actions:
the total mixed scalar curvature and the total extrinsic scalar curvature of a distribution.
The obtained Euler-Lagrange equations admit a number of solutions, e.g., twisted products, conformal submersions and isoparametric foliations. The paper generalizes recent results about the actions on codimension-one foliations for the case of arbitrary (co)dimension.
\end{abstract}



\section*{Introduction}\label{sec:intro}


Minimizing geometric quantities has been studied  for a
long time: recall, for example, isoperimetric inequalities and estimates of total curvature of submanifolds.
In the context of foliations and distributions, Gluck and Ziller \cite{gz1986}
studied the problem of minimizing functions like volume
defined for $k$-plane fields on a mani\-fold.
 In~all the cases mentioned above, they  consider a fixed Riemannian metric
and look for geometric objects (submanifolds, foliations) minimizing geometric
quantities defined usually as integrals of curvatures of different~types.

The following approach to problems in geometry of codimension-one foliations
is presented in \cite{rw-m}: given a foliated manifold and a property $Q$ of a submanifold,
depending on the principal curvatures of the leaves, study Riemannian metrics,
which minimize the integral of $Q$ in the class of variations of metrics, such that
the unit vector field orthogonal to the leaves is the same for all metrics of the variation family.
Geometric objects, as higher mean curvatures and scalar curvature type quantities,
have been exploited in order to embody complementary orthogonal distributions
into a theory aiming to find critical metrics for various actions.
Certainly (like in some of the cases mentioned before) such Riemannian structures may not exist, but if they do, they usually have interesting geometric properties and applications.

The~gravitational part of the Hilbert action is $J: g \mapsto \int_M \Sm(g)\,{\rm d}\vol_g$,
where $g$ is a metric of index~$1$, and $\Sm(g)$ is the {scalar curvature} of the spacetime~$(M,g)$.
The~integral is taken over $M$ if it converges;
otherwise,
one integrates over an arbitrarily large, relatively compact domain $\Omega$ in $M$, and it still provides the Einstein~equations.

Our objective is to develop variation formulas for the quantities of extrinsic geometry
for \textit{adapted} \textit{variations} of metrics on almost-product (e.g. foliated) pseudo-Riemannian manifolds,
and to apply them to study the Einstein-Hilbert type actions, see \eqref{E-Jmix} and \eqref{Jwidetildeex}.
These functionals are defined like the classical Einstein-Hilbert action,
the difference being the fact that the scalar curvature is replaced by the {mixed scalar curvature}
(i.e., an averaged mixed sectional curvature)
or the {extrinsic scalar curvature} of a non-degenerate distribution or foliation
-- the quantities which have been examined by several geometers, see \cite{bf,wa1} and bibliographies therein.
\textit{Adapted} \textit{variations} that we consider generalize the approach of \cite{rw-m}, to vary the metric in a way that preserves the almost-product structure of the manifold.
 We~deduce the Euler--Lagrange equations for an almost-product manifold and characterize the critical metrics in several classes of foliated manifolds. 
 The mixed Einstein-Hilbert action for a globally hyperbolic space\-time $(M^4,\,g)$ has been studied in~\cite{bdrs},
where the  Euler-Lagrange equa\-tions (called the mixed gravitational field equations) were derived
and their solutions for an empty space have been examined.
As~we shall see shortly,
the Euler-Lagrange equations for the Einstein-Hilbert type action involve several new tensors and
a new type of Ricci curvature (introduced in~\cite{r2010}, and studied
in \cite{bdr} for foliated closed Riemannian manifolds), whose properties need to be further investigated.

Our approach is based on variation formulas for the extrinsic geometry of foliations and almost-product manifolds
-- the quantities which can be expressed using configuration tensors
(i.e., the integrability tensor and the second fundamental form).
The paper develops methods of~\cite{rw-m}, where the variation formulas and functionals were studied
for codimension-one foliations; our main result in this case (Euler-Lagrange equations in Section~\ref{subsec:codim1fol}) coincides with an analogue of Einstein equations in \cite{bdrs}.
Our research poses open problems for further study, e.g. stability conditions of the action,
and the geometry of critical metrics with respect to adapted
variations of~metric.
Although adapted variations (of metric) preserve the orthogonal complement of a given distribution,
note that, unlike $\Sm_{\rm mix}$, the extrinsic scalar curvature does not depend expli\-citly on this complement. Therefore, in further work we shall also consider general variations more appropriate to this case.

The paper contains an introduction and two sections.
Section~\ref{sec:mixed-action} develops variation formulas for the quantities of extrinsic geometry
for {adapted} {variations} of metrics on almost-product (e.g. foliated) pseudo-Riemannian manifolds,
and applies them to study the total mixed scalar curvature 
and the total extrinsic scalar curvature of a distribution -- analogues of the classical Einstein-Hilbert action.
Its main goal are the Euler-Lagrange equations
for two types of adapted variations of metrics, the second of which preserves the volume of a domain $\Omega$ (and yields an analogue of Einstein equations with the cosmological constant).
Section~\ref{sec:main} is devoted to applications to foliated manifolds
including flows,
codimension-one foliations and conformal submersions with totally umbilical fibers.
We~
give examples (e.g. twisted products and isoparametric foliations) with sufficient conditions for critical metrics.

Throughout the paper everything (manifolds, distributions,
etc.) is assumed to be smooth
(i.e., $C^{\infty}$-differentiable) and oriented.
Following \cite{bf,O'Neill}, and in view of expected applications in theoretical physics,
we consider pseudo-Riemannian metrics.

\section{Einstein-Hilbert type action on almost-product manifolds}\label{sec:mixed-action}

A pseudo-Riemannian metric of index $q$ on $M$ is an element $g\in{\rm Sym}^2(M)$
 such that each $g_x\ (x\in M)$ is a \textit{non-degenerate bilinear form of index} $q$ on the tangent space $T_xM$.
When~$q=0$, i.e., $g_x$ is positive definite,
$g$ is a Riemannian metric (resp. a Lorentz metric~when $q=1$).
At~a point $x\in M$, a 2-dimensional linear subspace $X\wedge Y$ (called a plane section) of $T_xM$
is \textit{non-degenerate} if $W(X,Y):=g(X,X)\,g(Y,Y)-g(X,Y)\,g(X,Y)\ne0$.
For such section at $x$, the \textit{sectional curvature} is the number $K(X,Y)=g(R(X,Y)X, Y)/W(X,Y)$.
Here $R(X,Y)=\nabla_Y\nabla_X-\nabla_X\nabla_Y+\nabla_{[X,Y]}$ is the~curvature tensor
of the Levi-Civita connection $\nabla$ of $g$.

 The so called musical isomorphisms $\sharp$ and $\flat$ will be used for $(k,l)$-tensors, which form the
vector spaces $T^k_l M$ over $\RR$ and modules over $C^\infty(M)$.
For~example, if $\omega \in T^1_0 M$ is a 1-form and $X \in {\mathfrak X}_M,$ then
$\omega(Y)=g(\omega^\sharp,Y)$ and $X^\flat(Y) =g(X,Y)$ for any $Y\in {\mathfrak X}_M$.
For $(0,2)$-tensors $A$ and $B$ we have
 $\<A, B\> =\tr_g(A^\sharp B^\sharp)=\<A^\sharp, B^\sharp\>$.

\subsection{Preliminaries}\label{subsec:mixedEH}

A subbundle $\widetilde{\cal D}\subset TM$ (called a distribution) is non-degenerate,
if $\widetilde{\cal D}_x$ is a non-degenerate subspace of
$(T_x M,\, g_x)$ for every $x\in M$; in this case,
its complementary orthogonal distribution ${\cal D}$
(i.e., $\widetilde{\cal D}_x\cap\,{\cal D}_x=0$ and $\widetilde{\cal D}_x\oplus\,{\cal D}_x=T_xM$ for any $x\in M$)
is also non-degenerate.
Thus, we are entitled to consider a connected manifold $M^{n+p}$ with a pseudo-Riemannian metric $g$ and a pair
of complementary orthogonal non-degene\-rate distributions $\widetilde{\cal D}$ and ${\cal D}$ of ranks
$\dim_{\,\RR}\widetilde{\cal D}_x=n$ and $\dim_{\,\RR}{\cal D}_x=p$ for every $x\in M$
(called an \textit{almost-product structure} on $M$):
\begin{equation}\label{e:Decomposition}
 TM = \widetilde{\cal D} \oplus {\cal D}.
\end{equation}
The~following convention is adopted for the range of~indices:
\[
 a,b,\ldots{\in}\{1\ldots n\},\quad i,j,\ldots{\in}\{1\ldots p\}.
\]
 The~sectional curvature $K(X, Y)\ (X\in\widetilde{\cal D},\ Y\in{\cal D})$ is called {mixed}.
The function on~$M$,
\begin{equation}\label{eq-wal2}
 \Sm_{\rm mix} =\sum\nolimits_{\,a,i} K(E_a, {\cal E}_{i})
 =\sum\nolimits_{\,a,i}\epsilon_a \epsilon_i\,g(R(E_a, {\cal E}_{i})E_a,\, {\cal E}_{i})\,,
\end{equation}
where
 $\{E_a\subset\widetilde{\cal D},\,{\cal E}_{i}\subset{\cal D}\}$ is a local orthonormal frame
 and $\epsilon_i=g({\cal E}_{i},{\cal E}_{i}),\ \epsilon_a=g(E_a,E_a)$,
is the \textit{mixed scalar curvature}, see~\cite{wa1}.
 If~a distribution is spanned by a unit vector field $N$, i.e., $g(N,N)=\epsilon_N\in\{-1,1\}$,
 then $\Sm_{\rm mix} = \epsilon_N\Ric_{N}$, where $\Ric_{N}$ is the Ricci curvature in $N$-direction.
For surfaces foliated by curves, $\Sm_{\rm mix}$ is the Gaussian curvature.

Let $\mathfrak{X}_M$ (resp. $\mathfrak{X}_{\cal D}$ and $\mathfrak{X}_{\widetilde{\cal D}}$)
be the module over $C^\infty(M)$ of all vector fields on $M$ (resp. sections of ${\cal D}$ and $\widetilde{\cal D}$).
For every $X\in{\mathfrak X}_M$, let $\widetilde{X} \equiv X^\top$ be the $\widetilde{\cal D}$-component of $X$  (resp. $X^\perp$ the  ${\cal D}$-component of $X$) with respect to the
decomposition (\ref{e:Decomposition}).
Let ${\rm Sym}^2(M)$ be the space of all symmetric $(0,2)$-tensors tangent to~$M$.
A~tensor ${B} \in {\rm Sym}^2(M)$ is said to be {\em adapted} if ${B}(X^\top, Y^\perp)=0$ for any $X,Y\in{\mathfrak X}_M$.
Let~${\mathfrak M}\equiv{\mathfrak M}(\widetilde{\cal D},\,{\cal D})$ consist of all adapted symmetric tensors on $(M,\widetilde{\cal D},{\cal D})$.

 We study pseudo-Riemannian structures on a manifold $M$, minimizing the functional
\begin{equation}\label{E-Jmix}
 J_{\rm mix,\Omega}(g) : g \mapsto \int_{\Omega} \Sm_{\rm mix}(g)\, {\rm d} \vol_g
\end{equation}
for variations $g_t\ (g_0=g,\ |t|<\eps)$ preserving orthogonality of $\widetilde{\cal D}$ and ${\cal D}$,
i.e.,
\[
 g_t\in{\rm Riem}(M,\,\widetilde{\cal D},\,{\cal D}) := {\rm Riem}(M) \cap {\mathfrak M},
\]
where  ${\rm Riem}(M)$ is the subspace of pseudo-Riemannian metrics of given signature.
In all the paper, $\Omega$ in (\ref{E-Jmix}) is a relatively compact domain of $M$, containing supports of variations $g_t$.
Let ${\mathfrak M}_{\widetilde{\cal D}}$ and ${\mathfrak M}_{\cal D}$ be, respectively, the spaces of
symmetric $(0,2)$-tensors with the property ${B}(X^\perp,Y)=0$ (resp. ${B}({X}^\top,Y)=0$) for any $X,Y\in{\mathfrak X}_M$.
Then
\begin{equation}\label{e:Decomposition2}
 {\mathfrak M} = {\mathfrak M}_{\widetilde{\cal D}} \oplus {\mathfrak M}_{\cal D}\,,
\end{equation}
the decomposition is orthogonal with respect to the inner product $g^\ast$ induced on $\mathfrak M$
by a $g \in {\rm Riem}(M,\,\widetilde{\cal D},\, {\cal D})$.
For each $(0,2)$-tensor ${B}$ tangent to $M$ we define its
components $\widetilde{{B}}, \, {B}^\bot \in \Gamma(T^\ast M \otimes T^\ast M)$ by setting
$\widetilde {{B}}(X,Y) = {B}(X^\top,Y^\top)$ and ${B}^\perp(X,Y) = {B}(X^\perp, Y^\perp)$.
If~${B}\in{\rm Sym}^2(M)$ then ${B}\in{\mathfrak M} \Longleftrightarrow {B} = {B}^\bot + \widetilde{{B}}$,
see \eqref{e:Decomposition2}.
 In~particular, $g = g^\perp + \tilde{g}$ for any $g\in{\rm Riem}(M,\,\widetilde{\cal D},\,{\cal D})$.
 Note that if ${B}\in{\mathfrak M}$ then $\widetilde{\cal D}$ and ${\cal D}$ are ${B}^\sharp$-invariant.

Our purpose is to compute the directional derivatives
\begin{equation}\label{E-DJ}
 D_g J_{\rm mix,\Omega} : T_g\,{\rm Riem}(M,\,\widetilde{\cal D},\,{\cal D}) \equiv {\mathfrak M}\ \to\ {\mathbb R}
\end{equation}
for any $g \in {\rm Riem}(M,\,\widetilde{\cal D},\,{\cal D})$ on almost-product
or foliated manifolds $(M,\,\cal D,\,\widetilde{\cal D})$ and study the
curvature and the geometry of
$(M, g)$,
where $g$ is a critical point of $J_{\rm mix,\Omega}$ with respect to adapted variations
supported in $\Omega$. Certainly, we can restrict ourselves to the cases
$D_g J_{\rm mix,\Omega}:{\mathfrak M}_{\cal D}\to{\mathbb R}$
or $D_g J_{\rm mix,\Omega}:{\mathfrak M}_{\widetilde{\cal D}}\to{\mathbb R}$,
when $g$ is critical either for $g^\perp$-variations,
i.e., $D_g J_{\rm mix,\Omega}({B})=0$ for every ${B}\in{\mathfrak M}_{\cal D}$, or for $\tilde g$-variations,
i.e., $D_g J_{\rm mix,\Omega}({B})=0$ for every ${B}\in{\mathfrak M}_{\widetilde{\cal D}}$.

\smallskip

 We define several tensors for one of distributions, and introduce
similar tensors for the second distribution using $\,\widetilde{}\,$ notation.
 The~symmetric $(0,2)$-tensor ${r}_{\,\cal D}$, given~by
\begin{equation}\label{E-Rictop2}
 {r}_{{\cal D}}(X,Y) = \sum\nolimits_{a} \epsilon_a\,g(R(E_a, \, X^\perp)E_a, \, Y^\perp), \qquad  X,Y\in {\mathfrak X}_M,
\end{equation}
is referred to as the \textit{partial Ricci tensor} for $\cal D$.
In particular, by \eqref{eq-wal2},
\begin{equation}\label{E-Ric-Sm}
 \tr_{g}{r}_{\,\cal D}=\Sm_{\rm mix}.
\end{equation}
Note that the \textit{partial Ricci curvature} in the direction of
a unit vector $X\in{\cal D}$ is the ``mean value" of sectional curvatures over all mixed planes containing $X$.

Let~$T, h:\widetilde{\cal D}\times \widetilde{\cal D}\to{\cal D}$
be the integrability tensor and the second fundamental form
of $\widetilde{\cal D}$.
\begin{eqnarray*}
  T(X,Y)=(1/2)\,[X,\,Y]^\perp,\quad h(X,Y) \eq (1/2)\,(\nabla_X Y+\nabla_Y X)^\perp.
\end{eqnarray*}
Using the local orthonormal frame
$\{E_i,\,{\cal E}_a\}_{i\le p,\,a\le n}$, one may find the formulae
\begin{equation*}
 \<h,h\>=\!\sum\nolimits_{\,a,b}\epsilon_a\epsilon_b\,g(h({E}_a,{E}_b),h({E}_a,{E}_b)),\quad
 \<T,T\>=\!\sum\nolimits_{\,a,b}\epsilon_a\epsilon_b\,g(T({E}_a,{E}_b),T({E}_a,{E}_b)).
\end{equation*}
The mean curvature vector of $\widetilde{\cal D}$
is
 $H=\tr_{g} h=\sum\nolimits_a\epsilon_a h(E_a,E_a)$.
 The~distribution $\widetilde{\cal D}$
is called \textit{totally umbilical}, \textit{harmonic}, or \textit{totally geodesic}, if
 $h=\frac1nH\,\tilde g,\ H =0$, or $h=0$, respectively.


The Weingarten operator $A_Z$ of $\widetilde{\cal D}$ with respect to $Z\in{\cal D}$,
and the operator $T^\sharp_Z$ are defined~by
\[
 g(A_Z(X),Y)= g(h(X,Y),Z),\qquad g(T^\sharp_Z(X),Y)=g(T(X,Y),Z).
\]
The Divergence Theorem states that $\int_{M} (\Div\xi)\,{\rm d}\vol_g =0$, when $M$ is closed;
this is also true if $M$ is open and $\xi\in{\mathfrak X}_M$ is supported in a relatively compact domain
$\Omega\subset M$.
The~${\cal D}^\bot$-\textit{divergence} of $\xi$ is defined by
$\Div^\perp\xi=\sum\nolimits_{i} \epsilon_i\,g(\nabla_{i}\,\xi, {\cal E}_i)$.
 Thus, the divergence of $\xi$ is
\[
 \Div\xi=\tr(\nabla \xi) = \Div^\perp\xi + \widetilde{\Div}\,\xi.
\]
Recall that for a vector field $X\in{\mathfrak X}_{\cal D}$ and for the gradient and Laplacian of
$f\in C^2(M)$ we~have
\begin{eqnarray}\label{E-divN}
 {\Div}^\bot X \eq \Div X+g(X,\,H),\\
\nonumber
 g(\nabla f, X) \eq X(f),\quad  \Delta\,f =\Div(\,\nabla f).
\end{eqnarray}
Indeed, using $H=\sum\nolimits_{\,a\le n} \epsilon_a\,h(E_{a}, E_{a})$
and $g(X,\,E_{a})=0$, one derives~(\ref{E-divN}):
\begin{equation*}
 \Div X-{\Div}^\bot X =\sum\nolimits_{a} \epsilon_a\,g(\nabla_{E_{a}} X,\,E_{a})
 =-\sum\nolimits_{a}\epsilon_a\,g(h(E_{a}, E_{a}), X) = -g(X,\,H).
\end{equation*}
For a
$(1,2)$-tensor $P$ define a $(0,2)$-tensor
${\Div}^\bot P$ by
 $({\Div}^\bot P)(X,Y) = \sum\nolimits_i \epsilon_i\,g((\nabla_i\,P)(X,Y), {\cal E}_i)$.
Then the divergence of $P$ is $(\Div\,P)(X,Y) = \widetilde{\Div}\,P +{\Div}^\bot P$.
For a~${\cal D}$-valued $P$, similarly to \eqref{E-divN}, we have
$\sum\nolimits_a \epsilon_a\,g((\nabla_a\,P)(X,Y), E_a)=-g(P(X,Y), H)$ and
\begin{eqnarray}\label{E-divP}
 {\Div}^\bot P = \Div P+\<P,\,H\>\,,
\end{eqnarray}
where $\<P,\,H\>(X,Y):=g(P(X,Y),\,H)$ is a $(0,2)$-tensor.
For example,
 $\Div^\perp h = \Div h+\<h,\,H\>$.

To study the partial Ricci curvature (e.g. in Proposition~\ref{L-CC-riccati}) we introduce several tensors.

\begin{defn}\rm
 The ${\cal D}$-deformation ${\rm Def}_{\cal D}\,H$ of $H$ is the symmetric part of $\nabla H$
restricted to~${\cal D}$,
\[
 2\,{\rm Def}_{\cal D}\,H(X,Y)=g(\nabla_X H, Y) +g(\nabla_Y H, X),\quad X,Y\in {\mathfrak X}_{\cal D}.
\]
One may identify the antisymmetric part of $\nabla H$ restricted to~${\cal D}$,
regarded as a 2-form $d_{\cal D}\,H$,
\[
 2\,d_{\cal D}\,H(X,Y)=g(\nabla_X H, Y)-g(\nabla_Y H, X),\quad X,Y\in{\mathfrak X}_{\cal D}.
\]
Define self-adjoint $(1,1)$-tensors:
${\cal A}:=\sum\nolimits_{\,i}\epsilon_i A_{i}^2\,$
(called the \textit{Casorati operator} of ${\cal D}$) and
 ${\cal T}:=\sum\nolimits_{\,i}\epsilon_i(T_{i}^\sharp)^2$.
Define the symmetric $(0,2)$-tensor $\Psi$ by the identity
\begin{eqnarray*}
 \Psi(X,Y) \eq \tr (A_Y A_X+T^\sharp_Y T^\sharp_X),
 \quad X,Y\in{\mathfrak X}_{\cal D}\,.
\end{eqnarray*}
\end{defn}

 A.\,Gray~\cite{g1967} calculated the curvatures of the distributions $\widetilde{\cal D},\,{\cal D}$ from the curvature of $M$,
using configuration tensors.
These are analogues of the second fundamental form of a submanifold.
We~shall say that the \textit{extrinsic geometry of an almost-product structure} describes the properties,
which can be expressed using the configuration tensors.

\begin{prop}[see~\cite{bdr}]\label{L-CC-riccati}
Let $g\in{\rm Riem}(M,\,\widetilde{\cal D},\,{\cal D})$. Then the following identities hold:
\begin{eqnarray}\label{E-genRicN}
\nonumber
 r_{{\cal D}} \eq \Div\tilde h +\<\tilde h,\,\tilde H\>
  -\widetilde{\cal A}^\flat -\widetilde{\cal T}^\flat
 -\Psi +{\rm Def}_{\cal D}\,H\,,\\
  d_{\cal D}\,H  \eq -\,\widetilde{\Div}\,\tilde T +\sum\nolimits_a\epsilon_a
 \big(\tilde A_{a}\tilde T^\sharp_{a} +\tilde T^\sharp_{a}\tilde A_{a}\big)^\flat.
\end{eqnarray}
\end{prop}


The
\textit{extrinsic
curvature} of $\widetilde{\cal D}$,
\[
 R_{\,\rm ex}(X,\,Y,\,Z,\,W)=
 g(h(X^\top,Z^\top),h(Y^\top,W^\top))-g(h(X^\top,\,Y^\top),h(Z^\top,\,W^\top))
\]
is useful in the  study of extrinsic geometry of foliations, see \cite{rw-m}.
The traces (along $\widetilde{\cal D}$)
\begin{eqnarray*}
 \Ric_{\rm ex}(X,\,Y)\eq
 \sum\nolimits_{a} \epsilon_a\,R_{\,\rm ex}(X,\,Y,\,E_a,\,E_a),\\
 \Sm_{\,\rm ex} \eq
 \sum\nolimits_{a}\epsilon_a\Ric_{\rm ex}(E_a,\,E_a)\,.
\end{eqnarray*}
are the extrinsic Ricci and scalar curvatures of $\widetilde{\cal D}$.
Note that
 $\Sm_{\,\rm ex}=g(H,H)-\<h,h\>$.

\begin{rem}\rm
Tracing \eqref{E-genRicN}$_1$ over ${\cal D}$ and applying \eqref{E-Ric-Sm} and the equalities
\begin{eqnarray*}
 \tr_{g}\Psi \eq \sum\nolimits_{\,i}\epsilon_i\tr_{g}(A_i^2 + (T^\sharp_i)^2) = \<h,h\> - \<T,T\>,\\
 \tr{\cal A}\eq\<h,h\>,\quad \tr{\cal T} = -\<T,T\>,\\
 \tr_{g}\,({\Div}\,h) \eq \Div H ,\quad
 \tr_{g}\,({\rm Def}_{\cal D}\,H) = \Div H +g(H,H)
\end{eqnarray*}
yield the formula (see also \cite{wa1})
\begin{equation}\label{eq-ran-ex}
 \Sm_{\rm mix} = \Sm_{\,\rm ex} +\widetilde\Sm_{\,\rm ex} +\<T,T\> +\<\tilde T,\tilde T\> + \Div(H+\tilde H)\,,
\end{equation}
which shows that $\Sm_{\rm mix}$ is built of the invariants of the  extrinsic geometry of the distributions.
\end{rem}

\subsection{Variation formulas}
\label{sec:prel}

Given an adapted pseudo-Riemannian metric $g$ on $(M,\widetilde{\cal D},{\cal D})$,
consider smooth $1$-parameter variations of $g_0 = g$,
\begin{equation}\label{E-Sdtg}
 \big\{ g_t \in {\rm Riem}(M, \, \widetilde{\cal D}, \, {\cal D}) : |t| < \eps \big\}\,.
\end{equation}
The induced infinitesimal variations, presented by a symmetric $(0,2)$-tensor
${B}_t\equiv(\partial g_t/\partial t)\in{\mathfrak M}$, are supported in a relatively compact domain $\Omega$ in $M$.
 We adopt the notations
\begin{equation}\label{E-Sdtg-2}
 \partial_t \equiv \partial/\partial t,\quad
 {B} \equiv {\dt g_t}_{\,|\,t=0}.
\end{equation}
We will define several tensors for one of distributions,
similar notions for the second distribution are introduced using $\,\widetilde{}\,$ notation.
Taking into account
(\ref{e:Decomposition2}), it is sufficient  to work with special curves
$\{g_t\}_{|t|<\eps}$ starting at $g \in {\rm Riem}(M,\,\widetilde{\cal D},\,{\cal D})$
called $g^\perp$-\textit{variations}:
\begin{equation}\label{e:Var}
 \big\{ g^\bot_t +\tilde g : \ |t| < \,\eps \big\},
\end{equation}
as the associated infinitesimal variations ${B}_t$ lie in ${\mathfrak M}_{\cal D}$.
For adapted variations \eqref{E-Sdtg}--\,\eqref{E-Sdtg-2}
we have, see for example~\cite{rw-m},
\begin{equation}\label{eq2G}
 2\,g_t(\dt(\nabla^t_X\,Y), Z) = (\nabla^t_X\,{B})(Y,Z)+(\nabla^t_Y\,{B})(X,Z)-(\nabla^t_Z\,{B})(X,Y),\quad
 X,Y,Z\in\mathfrak{X}_M.
\end{equation}

\begin{lem}\label{prop-Ei-a}
Let a local $(\widetilde{\cal D},\,{\cal D})$-adapted frame $\{E_a,\,{\cal E}_{i}\}$
evolve by \eqref{E-Sdtg}--\,\eqref{E-Sdtg-2} according to
 \begin{equation*}
 \dt E_a=-(1/2)\,{B}_t^\sharp(E_a),\qquad
 \dt {\cal E}_{i}=-(1/2)\,{B}_t^\sharp({\cal E}_{i}).
\end{equation*}
 Then,  for all $\,t, $ $\{E_a(t),{\cal E}_{i}(t)\}$ is a $g_t$-orthonormal frame adapted to
 $(\widetilde{\cal D},{\cal D})$.
\end{lem}

\proof
For $\{E_a(t)\}$ (and similarly for $\{{\cal E}_{i}(t)\}$) we have
\begin{eqnarray*}
 &&\dt(g_t(E_a, E_b)) = g_t(\dt E_a(t), E_b(t)) +g_t(E_a(t), \dt E_b(t))
 +(\dt g_t)(E_a(t), E_b(t))  \\
 &&= {B}_t(E_a(t), E_b(t))-\frac12\,g_t({B}_t^\sharp(E_a(t)), E_b(t)) -\frac12\,g_t(E_a(t), {B}_t^\sharp(E_b(t)))=0.
 \qed
\end{eqnarray*}

\begin{lem}[see~\cite{bdr}]\label{L-tildeHh}
For $g^\perp$-variations \eqref{E-Sdtg}--\,\eqref{E-Sdtg-2} we have
\begin{eqnarray}\label{eq-hatbH-1}
 2\,\dt\tilde h(X,Y) \eq (\tilde h-\tilde T)({B}^\sharp(X),Y)  +(\tilde h+\tilde T)(X,{B}^\sharp(Y))
 -\!\widetilde\nabla{B}(X,Y),\\
\label{eq-hatH}
 2\,\dt\tilde H \eq -\!\widetilde\nabla(\tr {B}^\sharp),\qquad
 \dt h = -{B}^\sharp\circ h,\qquad \dt H = -{B}^\sharp(H).
\end{eqnarray}
Hence, $g^\perp$-variations preserve total umbilicity, total geodesy and harmonicity of $\,\widetilde{\cal D}$.
\end{lem}

 Define symmetric $(0,2)$-tensors $\Phi_h$ and $\Phi_T$ (the last one vanishes when $n=1$), using the identities
 (with arbitrary~$B\in{\mathfrak M}$)
\begin{eqnarray*}
 \<\Phi_h,\ B \> \eq B(H,\,H) -\sum\nolimits_{\,a,\,b} \epsilon_a\,\epsilon_b\,B(h(E_a,E_b), h(E_a,E_b)),\\
 \<\Phi_T,\ B \> \eq -\sum\nolimits_{\,a,\,b} \epsilon_a\,\epsilon_b\,B(T(E_a,E_b), T(E_a,E_b)).
\end{eqnarray*}
We have
 $\tr_{g}\Phi_h=\Sm_{\,\rm ex}$ and $\tr_{g}\Phi_T=-\<T,T\>$.
Define a
$(1,1)$-tensor (with zero trace)
\[
 {\cal K} =\sum\nolimits_{\,i} \epsilon_i\,[T^\sharp_i , A_i]
 =\sum\nolimits_{\,i} \epsilon_i\,(T^\sharp_i A_i - A_i T^\sharp_i).
\]

\begin{rem}\rm \label{Phihzero}
1) Let $g$ be definite on $\widetilde{\cal D}$. Then $\Phi_{h} =0$ if and only if one of the following holds:
\[
 (i)~h=0; \quad
 (ii)~H\ne0,\ \Sm_{\,\rm ex}=0 \ \mbox{\rm and the image of}\ h\ \mbox{\rm is spanned by} \ H\,.
\]
To show this, consider any vector $X \in {\cal D}$ such that $g(X,H)=0$. Then
\begin{equation*}
 \<\Phi_{h}, X^\flat\otimes X^\flat\> = g(X,H)^{2} -\sum\nolimits_{\,a,b}\epsilon_{a}\epsilon_{b}\,g(X, h(E_{a},E_{b}))^{2}
 = -\sum\nolimits_{\,a,b}\epsilon_{a}\epsilon_{b}\,g(X, h(E_{a}, E_{b}))^{2}.
\end{equation*}
Since all $\epsilon_{a}$ are of the same sign,
the above sum is equal to zero if and only if every summand vanishes.
Moreover, $\<\Phi_{h},\,H^{\flat}\otimes H^{\flat}\> = g(H,H)\,\Sm_{\,\rm ex}$ holds.
Similarly, if $\Phi_{T} = 0$ then we have
\[
 \<\Phi_{T}, X^\flat\otimes X^\flat\> = -\sum\nolimits_{\,a,b}\epsilon_{a}\,\epsilon_{b}\,g(X, T(E_{a},E_{b}))^{2}=0
 \quad (X\in{\cal D}).
\]
Hence, if $g$ is definite on $\widetilde{\cal D}$ ($\epsilon_{a}=\epsilon_{b}$)
then the condition $\Phi_{T}=0$ is equivalent to $T=0$.
Therefore, $\Phi_T$ can be viewed as a measure of non-integrability of ${\cal D}$.

2) If ${\cal D}$ is integrable then $\tilde T^\sharp_a = 0$ for all $a \in \{1, \ldots, n\}$, hence $\tilde{\cal K} =0$. Also, if ${\cal D}$ is totally umbilical, then every operator $\tilde A_a$ is a multiple of identity and $\tilde{\cal K}$ vanishes as well.
\end{rem}

\begin{lem}\label{L-H2h2-D}
For $g^\perp$-variations
 we have
\begin{eqnarray}\label{E-h2T2-D1}
 &&\dt\,\widetilde{\Sm}_{\,\rm ex} = \<(\Div\tilde H)\,g^\bot -\Div\tilde h -\tilde{\cal K}^\flat,\, {B}\>
 +\Div(\<\tilde h,\,{B}\>-(\tr_{g} {B})\tilde H),\\
\label{E-h2-D1}
 &&\dt\,\Sm_{\,\rm ex} = -\<\Phi_h,\ {B}\>,\\
\label{E-T2-D1}
 && \dt\,\<\tilde T,\tilde T\> = \<2\,\widetilde{\cal T}^\flat,\ {B}\>,\quad
    \dt\,\<T,T\> =-\<\Phi_T,\ {B}\>\,.
\end{eqnarray}
\end{lem}

\proof
Assume $\nabla_{a}\,{\cal E}_i\in\widetilde{\cal D}_x$ at a point $x\in M$.
In the calculations below we use \eqref{eq2G} and Lemmas~\ref{prop-Ei-a} and~\ref{L-tildeHh}.
First we obtain~\eqref{E-T2-D1}$_1$:
\begin{eqnarray*}
 &&\dt\,\<\tilde T,\tilde T\> = 2\sum\nolimits_{\,i,j,a} \epsilon_i\,\epsilon_j\,\epsilon_a\, g(\tilde T({\cal E}_i,{\cal E}_j), E_a)
 \,g\big(\tilde T({\dt\cal E}_i,{\cal E}_j) +\tilde T({\cal E}_i,\dt{\cal E}_j), E_a\big)  \\
 \eq -\sum\nolimits_{\,i,j,a} \epsilon_i\,\epsilon_j\,\epsilon_a\, g(\tilde T({\cal E}_i,{\cal E}_j), E_a)
 \,g\big(\tilde T({B}^\sharp({\cal E}_i),{\cal E}_j)+\tilde T({\cal E}_i,{B}^\sharp({\cal E}_j)), E_a\big)  \\
 \eq -\!\sum\nolimits_{\,i,j,a} \epsilon_i\,\epsilon_j\,\epsilon_a\, g(\tilde T^\sharp_a({\cal E}_i),{\cal E}_j)
 \,g((\tilde T^\sharp_a {B}^\sharp +{B}^\sharp\tilde T^\sharp_a)({\cal E}_i), {\cal E}_j) \\
 \eq -\!\sum\nolimits_{\,i,a} \epsilon_i\,\epsilon_a\, g(
 (\tilde T^\sharp_a {B}^\sharp{+}{B}^\sharp\tilde T^\sharp_a)({\cal E}_i), \tilde T^\sharp_a({\cal E}_i))
 =\sum\nolimits_{\,i,a} \epsilon_i\,\epsilon_a\, g(
 ((\tilde T^\sharp_a)^2{B}^\sharp+\tilde T^\sharp_a {B}^\sharp\tilde T^\sharp_a)({\cal E}_i),\,{\cal E}_i) \\
 \eq 2\sum\nolimits_{\,a} \epsilon_a\, \tr ((\tilde T^\sharp_a)^2 {B}^\sharp)
 =2\tr (\widetilde{\cal T} {B}^\sharp) =\<2\,\widetilde{\cal T}^\flat,\ {B}\>.
\end{eqnarray*}
Next, by \eqref{E-divP}
we obtain
\begin{eqnarray*}
 &&\dt\,\<\tilde h,\tilde h\> = 2\sum\nolimits_{\,i,j,a}\epsilon_i\,\epsilon_j\,\epsilon_a\,g(\tilde h({\cal E}_i,{\cal E}_j),\,E_a)
 g(\dt(\tilde h({\cal E}_i,{\cal E}_j)),\,E_a) \\
 \eq 2\sum\nolimits_{\,i,j,a} \epsilon_i\,\epsilon_j\,\epsilon_a\, g(\tilde h({\cal E}_i,{\cal E}_j),\, E_a)
 \,g\big((\dt\tilde h)({\cal E}_i,{\cal E}_j)
 +\tilde h(\dt{\cal E}_i,{\cal E}_j)
 +\tilde h({\cal E}_i,\dt{\cal E}_j),\,E_a\big) \\
  \eq\!\sum\nolimits_{\,i,j,a} \epsilon_i\,\epsilon_j\,\epsilon_a\, g(\tilde h({\cal E}_i,{\cal E}_j),\, E_a)
 \,\big(g(\tilde T({\cal E}_i,{B}^\sharp({\cal E}_j))-\tilde T({B}^\sharp({\cal E}_i),{\cal E}_j),\,E_a) -\nabla_a\,{B}({\cal E}_i,{\cal E}_j)\big)  \\
  \eq \sum\nolimits_{\,i,j,a} \epsilon_i\,\epsilon_j\,\epsilon_a \Big(
 g(\tilde A_a({\cal E}_i),\,{\cal E}_j)\,g([{B}^\sharp, \tilde T^\sharp_{a}]({\cal E}_i),\,{\cal E}_j)
 -\nabla_a\,\big({B}({\cal E}_i,{\cal E}_j)\,g(\tilde h({\cal E}_i,{\cal E}_j),\, E_a)\big)   \\
 \minus \nabla_a\,g(\tilde h({\cal E}_i,{\cal E}_j),\, E_a)\,{B}({\cal E}_i,\,{\cal E}_j)\Big)
 = \sum\nolimits_{\,i,a}\epsilon_i\,\epsilon_a\big({B}(\tilde T^\sharp_{a}({\cal E}_i), \tilde A_a({\cal E}_i))
 +{B}({\cal E}_i,\,\tilde T^\sharp_{a}\tilde A_a({\cal E}_i))\big) \\
 \plus \<\widetilde{\Div}\,\tilde h - \<\tilde h,\,\tilde H\>,\, {B}\> -\Div(\<\tilde h,\,{B}\>)
 = \<
 {\Div}\,\tilde h +\tilde{\cal K}^\flat,\, {B}\> -\Div\<\tilde h,\,{B}\>.
\end{eqnarray*}
Here we used $(\tilde T^\sharp_{a})^*=-\tilde T^\sharp_{a}$, $(\tilde A_{a})^*=\tilde A_{a}$ and $(B^\sharp)^* = B^\sharp$, hence
\[
  \tr(\tilde T^\sharp_{a} \tilde A_a B^\sharp) = \tr(B^\sharp(\tilde T^\sharp_{a} \tilde A_a)^*)
 =\tr((\tilde T^\sharp_{a} \tilde A_a)^*B^\sharp) =\tr(\tilde A_a(\tilde T^\sharp_{a})^*B^\sharp)
 = -\tr(\tilde A_a \tilde T^\sharp_{a} B^\sharp).
\]
Next, we get \eqref{E-h2T2-D1}, applying ${B}(\tilde H,\tilde H)=0$ (since ${B}$ vanishes on $\widetilde{\cal D}$) and
\begin{eqnarray*}
 \dt\,g(\tilde H,\tilde H) \eq 2\,g(\dt\tilde H,\,\tilde H) =-g(\nabla(\tr {B}^\sharp),\,\tilde H).
\end{eqnarray*}
Notice that $g(\nabla(\tr {B}^\sharp),\,\tilde H)=\Div((\tr {B}^\sharp)\tilde H)-(\Div\tilde H)\tr {B}^\sharp$.
We have
\begin{eqnarray*}
 \dt\,g(H,H) \eq {B}(H,H) + 2\,g(\dt H,\,H) = {B}(H,H) -2\,g({B}^\sharp(H),\,H) =-{B}(H,H), \\
 \dt\,\<h,h\>\eq\dt\sum\nolimits_{\,i,\,a,\,b}\epsilon_i\,\epsilon_a\,\epsilon_b\,g(h(E_a,\,E_b),\,{\cal E}_i)^2 \\
 \eq 2\sum\nolimits_{\,i,\,a,\,b}\epsilon_i\,\epsilon_a\,\epsilon_b\,g(h(E_a,\,E_b),\,{\cal E}_i)\,\dt g(h(E_a,\,E_b),\,{\cal E}_i) \\
 \eq -\!\sum\nolimits_{\,i,\,a,\,b} \epsilon_i\,\epsilon_a\,\epsilon_b\,g(h(E_a,\,E_b),\,{\cal E}_i)\,
 g(h(E_a,\,E_b),\,{B}^\sharp({\cal E}_i))  \\
 \eq -\!\sum\nolimits_{\,a,\,b} \epsilon_a\,\epsilon_b\, {B}(h(E_a,\,E_b),\,h(E_a,\,E_b))\,.
\end{eqnarray*}
From the above, \eqref{E-h2-D1} follows. Finally, we have \eqref{E-T2-D1}$_2$:
\begin{eqnarray*}
 \dt\,\<T,T\> \eq \dt\sum\nolimits_{\,i,\,a,\,b} \epsilon_i\,\epsilon_a\,\epsilon_b\, g(T(E_a,\,E_b),\,{\cal E}_i)^2
 \\ \eq 2\sum\nolimits_{\,i,\,a,\,b} \epsilon_i\,\epsilon_a\,\epsilon_b\, g(T(E_a,\,E_b),\,{\cal E}_i)\,\dt(g(T(E_a,\,E_b),\,{\cal E}_i))  \\
 \eq 2\sum\nolimits_{\,i,\,a,\,b} \epsilon_i\,\epsilon_a\,\epsilon_b\, g(T(E_a,\,E_b),\,{\cal E}_i)\,
 \big( {B}(T(E_a,\,E_b),\,{\cal E}_i) +g(T(E_a,\,E_b),\,\dt{\cal E}_i)\big)   \\
 \eq \sum\nolimits_{\,i,\,a,\,b} \epsilon_i\,\epsilon_a\,\epsilon_b\, g(T(E_a,\,E_b),\,{\cal E}_i)\,g(T(E_a,\,E_b),\,{B}^\sharp({\cal E}_i)) \\
 \eq \sum\nolimits_{\,a,\,b} \epsilon_a\,\epsilon_b\, {B}(T(E_a,\,E_b),\,T(E_a,\,E_b)).\quad  \qed
\end{eqnarray*}

\subsection{Euler-Lagrange equations}
\label{subsec:EU-general}

In this section we derive directional derivatives \eqref{E-DJ}
and the Euler-Lagrange equations of $J_{\rm mix,\Omega}$
on an open pseudo-Rieman\-nian almost-product manifold for two types of $g^\perp$-variations
(i.e., either preserving the volume or not).
 For arbitrary $f \in L^1 ({\Omega},\; {\rm d}\, {\rm vol}_g)$ denote~by
\[
 f({\Omega}, g) = {\rm Vol}^{-1}({\Omega}, g) \int_{\Omega} f\,{\rm d}\,{\rm vol}_g
\]
the mean value of $f$ on ${\Omega}$.
 Together with a family $g_t$ of \eqref{e:Var}, consider on $\Omega$ the metrics
\begin{eqnarray}\label{e:Var-Bar}
 \bar g_t \eq \phi_t g^\perp_t + \tilde g, \;\;\;
 \phi_t \equiv \big({\rm Vol}({\Omega}, g_t)/{\rm Vol}({\Omega}, g)\big)^{-2/p}\,,\ \ |t|<\eps.
\end{eqnarray}
Recall, see \cite{rw-m}, that the volume form evolves as
\begin{equation}\label{E-dtdvol}
 \partial_t\,\big({\rm d}\vol_{g_t}\!\big) = \frac12\,(\tr_{g_t} {B}_t)\,{\rm d}\vol_{g_t}.
\end{equation}
We will show that ${\rm Vol}({\Omega},\bar g_t) = {\rm Vol}({\Omega}, g)$ for all $t$.
As $\bar g_t$ are ${\cal D}$-conformal to $g_t$ with constant scale $\phi_t$, their volume forms are related~as
\begin{equation}\label{e:Conf-Vol}
 {\rm d}\vol_{\,\bar g_t} = \phi^{p/2}_t{\rm d}\vol_{\,g_t};
\end{equation}
hence, ${\rm Vol}({\Omega},\bar g_t) = \int_{\,\Omega} {\rm d}\vol_{\,\bar g_t} = {\rm Vol}({\Omega}, g)$.
 Let us differentiate (\ref{e:Conf-Vol}) in order to obtain
\begin{equation*}
 \dt\;({\rm d}\vol_{\,\bar g_t}) = (\phi_t^{p/2})^\prime \; {\rm d}\vol_{g_t} + \phi_t^{p/2}\;
 \dt\,({\rm d}\vol_{g_t})
  = \frac12\,\big(\tr {B}^\sharp_t -(\tr_{g_t} {B}_t)({\Omega}, g_t)
 \big)\,{\rm d}\vol_{\,\bar g_t}\,.
\end{equation*}
We have used (\ref{E-dtdvol}) and the fact that $\phi_0=1$ and
\begin{equation}\label{E-2vols}
 \phi^\prime_t = -\frac2p\,\Big(\frac{{\rm Vol}({\Omega},g_t)}{{\rm Vol}({\Omega}, g)}\Big)^{-\frac2p-1}
 \frac{1}{{\rm Vol}({\Omega}, g)} \int_{\Omega} \dt({\rm d}\vol_{\,g_t})
 =-\frac{\phi_t}{p}\,(\tr_{g_t} {B}_t)({\Omega}, g_t).
\end{equation}

Next we give several technical lemmas.

\begin{lem}\label{L-div}
For all $g^\perp$-variations \eqref{E-Sdtg}--\,\eqref{E-Sdtg-2} and all $g^\perp$-variations preserving the volume of $\Omega$
the evolution of $\Div$ on a $t$-dependent vector field $X$ is given by the formula
\begin{equation}\label{dtdiv}
 \dt(\,\Div X) = \Div(\dt\,X) +
 (1/2)\,X(\tr {B}^\sharp).
\end{equation}
\end{lem}

\begin{proof}
First, consider arbitrary $g^\perp$-variation $g_t$.
Differentiating the formula
 $\Div X \cdot {\rm d}\vol_g = {\cal L}_{X}({\rm d}\vol_g)$,
see \cite{O'Neill},
we obtain (\ref{dtdiv}).
Observe that for $g^\perp$-variations preserving the volume of $\Omega$, the divergence $\Div_{\bar{g}}$ with respect to metric $\bar{g} = \phi\,g^{\bot} + \tilde g$ is given~by
\begin{eqnarray*}
\nonumber
 \Div_{\bar{g}} X \eq \sum\nolimits_{a} \epsilon_{a}\,\tilde g({\bar \nabla}_{a}X, E_{a})
 +\sum\nolimits_{i} \epsilon_{i}\,\phi\,g^{\bot}({\bar \nabla}_{\phi^{-1/2} {\cal E}_{i}}X, \phi^{-1/2} {\cal E}_{i})\\
 \eq \sum\nolimits_{a} \epsilon_{a}\,g({\nabla}_{a}X, E_{a})
 + \sum\nolimits_{i} \epsilon_{i}\,g({ \nabla}_{ {\cal E}_{i}}X,  {\cal E}_{i}) = \Div X.
\end{eqnarray*}
Hence, again we obtain (\ref{dtdiv}).
\end{proof}

\begin{lem}\label{divH+tildeH}
For any $g^\perp$-variation $g_t$ and $\bar g_t$ of \eqref{e:Var-Bar} supporting in
$\Omega\subset M$, we have
\begin{equation*}
 {\rm\frac{d}{dt}}\int_{\Omega} \Div( H + \tilde H)\, {\rm d}\vol_g \!=\!\left\{
 \begin{array}{cc}
   0 & \mbox{\rm for}\ \ g_t, \\
   \!\!\frac{1}{2}\Div\big( \frac{2-p}p\,H - {\tilde H} \big)(\Omega, g) \int_{\Omega} (\tr_g B) \, {\rm d}\vol_g &
   \mbox{\rm for}\ \ \bar g_t.
 \end{array}\right.
\end{equation*}
\end{lem}

\proof Using the equations for time derivatives of mean curvatures and the volume form, we~get
\begin{eqnarray*}
&& {\rm\frac{d}{dt}}\int_{\Omega} \Div( H + \tilde H) \, {\rm d}\vol_g = \int_{\Omega} \dt\, (\Div(H + \tilde H)) \, {\rm d}\vol_g
 +\int_{\Omega} \Div(H + \tilde H) \, \dt\, ({\rm d}\vol_g) \\
 && = -\int_{\Omega}\Div({B}^\sharp(H))\,{\rm d}\vol_g
 + \int_{\Omega}\Div(-\widetilde\nabla(\tr {B}^\sharp))\,{\rm d}\vol_g \\
 && +\,\frac{1}{2}\int_{\Omega}\Big( \Div((\tr {B}^\sharp)(H + \tilde H))-\Div(H+\tilde H)(\tr {B}^\sharp)
 + (\tr {B}^\sharp) \Div(H + \tilde H) \Big)\, {\rm d}\vol_g \\
 && = \int_{\Omega} \Big(-\Div(\widetilde\nabla(\tr {B}^\sharp) )
 - \Div ( {B}^\sharp(H) ) + \frac{1}{2}\Div \big( (\tr {B}^\sharp) (H + \tilde H) \big) \Big)\, {\rm d}\vol_g = 0,
 \end{eqnarray*}
since all the above terms are integrals of divergences of vector fields supported in $\Omega$.

For $g^\perp$-variations preserving the volume of $\Omega$,
all the following derivatives with respect to $t$ will be calculated at $t=0$.
By Lemma~\ref{L-div} and using \eqref{E-2vols} we have
\[
 \dt(\Div H_{\bar{g}}) = \dt\,(\Div H)
 + \frac{1}{p}\,(\Div H)\cdot(\tr_g B)(\Omega,g)\,,
\]
while $\tilde H_{\bar{g}} = \tilde H$, see \cite{bdrs}, and hence $\dt(\Div\tilde  H_{\bar{g}}) = \dt\,(\Div\tilde  H)$.
We also have
\[
 \dt\,({\rm d}\vol_{\bar{g}}) = \dt\, ({\rm d}\vol_g)
 - \frac{1}{2}\,(\tr_g B)(\Omega,g)\,{\rm d}\vol_{\bar{g}}\,.
\]
Thus,
\begin{eqnarray*}
 {\rm\frac{d}{dt}}\int_{\Omega} \Div( H_{\bar{g}} + \tilde H_{\bar{g}})\,{\rm d}\vol_{\bar{g}}
 \eq \int_{\Omega} \dt\,(\Div( H + \tilde H) ) \, {\rm d}\vol_{g}
 +\int_{\Omega} \Div (H + \tilde H) \, \dt\,( {\rm d}\vol_{g})  \\
 \eq (\tr_g B)(\Omega,g)\int_{\Omega}\Div\Big(\frac{2-p}{2\,p}\,H -\frac{1}{2}\,{\tilde H}\Big)\,{\rm d}\vol_{g}.
 \qed
\end{eqnarray*}

\begin{prop}\label{P-1-5}
The
$g^\perp$-variations of metric for the action \eqref{E-Jmix} associated with $\bar g_t$ and $g_t$ are related~by
\begin{eqnarray}\label{E-varJh-f}
 {\rm\frac{d}{dt}}\,J_{\rm mix,\Omega}(\bar{g}_t)_{\,|\,t=0}
 ={\rm\frac{d}{dt}}\,J_{\rm mix,\Omega}({g}_t)_{\,|\,t=0}
 - \frac12\,\Sm^*_{\rm mix}(\Omega,g) \int_{\Omega} \big(\tr_{g} {B} \big)\,{\rm d}\vol_g \,,
\end{eqnarray}
where
\begin{equation}\label{E-S-star}
 \Sm^*_{\rm mix} = \Sm_{\rm mix} -
 \frac2p\,(\Sm_{\,\rm ex} +2\,\<\tilde T,\tilde T\> -\<T,T\> + \Div H).
\end{equation}
\end{prop}

\proof Let us fix a $g^\perp$-variation $g_t$, see \eqref{e:Var}$_1$.
 By \eqref{eq-ran-ex} and Lemma \ref{divH+tildeH}, we have
\begin{equation*}
 {\rm\frac{d}{dt}}\,J_{\rm mix,\Omega}(g_t)
 = {\rm\frac{d}{dt}} \int_{\Omega} Q(g_t)\ {\rm d}\vol_{g_t}
 \,,
\end{equation*}
where $Q(g):=\Sm_{\rm mix} -\Div(H+\tilde H)$ is represented using \eqref{eq-ran-ex} as
\begin{equation}\label{E-Q-def}
 Q(g) = \Sm_{\,\rm ex}(g) +\widetilde\Sm_{\,\rm ex}(g) +\<\tilde T,\tilde T\>_{g} +\<T,T\>_{g}\,.
\end{equation}
For $\bar g_t = \phi_t g_t + \tilde g$, see \eqref{e:Var-Bar}, we have, see \cite{bdrs},
\begin{eqnarray*}
 && H_{\bar g} = \phi^{-1} H , \quad
 {\tilde H}_{\bar g} = {\tilde H}, \quad
 h_{\bar g} = \phi^{-1} h , \quad
 {\tilde h}_{\bar g} = \phi\,\tilde h, \\
 &&\<T,T\>_{\bar g}=\phi\,\<T,T\>_{g},\quad
 \<h_{\bar g},h_{\bar g}\>_{\bar g}=\phi^{-1}\<h,h\>_{g},\quad
 \<\tilde h_{\bar g},\tilde h_{\bar g}\>_{\bar g}=\<\tilde h,\tilde h\>_{g},\\
 && \<\tilde T,\tilde T\>_{\bar g}=\phi^{-2}\<\tilde T,\tilde T\>_{g},\quad
 {\bar g}(H_{\bar g},H_{\bar g})=\phi^{-1} g(H,H),\quad
 {\bar g}(\tilde H_{\bar g},\tilde H_{\bar g})= g(\tilde H,\tilde H),
\end{eqnarray*}
where subscript $\bar g$ corresponds to geometric quantities calculated with respect to $\bar g$.
Hence,
\[
 Q(\bar g_t) = Q(g_t) +(\phi^{-1}_t-1)\,\Sm_{\,\rm ex}(g_t) +(\phi^{-2}_t-1)\,\<\tilde T,\tilde T\>_{g_t} +(\phi_t-1)\,\<T,T\>_{g_t}\,.
\]
Differentiating the above at $t=0$ and using $\phi_0=1$, we get
\begin{eqnarray*}
 \dt Q(\bar g_t)_{|\,t=0} \eq \dt Q(g_t)_{|\,t=0}
 -\phi^{\,\prime}_0\,\big(\Sm_{\,\rm ex}(g) +2\,\<\tilde T,\tilde T\>_{g} -\<T,T\>_{g}\big)\,,
\end{eqnarray*}
where $\phi^{\,\prime}_0=-\frac{1}p\,(\tr_g B)(\Omega,g)$, see \eqref{E-2vols}.
Using Lemma~\ref{divH+tildeH} we obtain
\begin{eqnarray}\label{E-varJh-init}
\nonumber
 {\rm\frac{d}{dt}}\,J_{\rm mix,\Omega} ({g}_t)_{\,|\,t=0}
 \eq \int_{\Omega}\!\Big\{\dt Q(g_t)_{|\,t=0} +\frac12\,Q(g)\tr_{g} {B}\Big\}\,{\rm d}\vol_g\,,\\
 {\rm\frac{d}{dt}}\,J_{\rm mix,\Omega} (\bar{g}_t)_{\,|\,t=0}
\nonumber
  \eq \int_{\Omega}\!\Big\{\dt Q(\bar g_t)_{|\,t=0} +\frac12\,Q(g)(\tr_{g} {B} + p\,\phi^{\,\prime}_0)\Big\}\,{\rm d}\vol_g\\
 \plus{\rm\frac{d}{dt}}\int_{\Omega} \Div (H_{\bar g_t} + \tilde H_{\bar g_t})\,{\rm d}\vol_{\bar g_t}\,\!_{|\,t=0}.
\end{eqnarray}
Hence,
\begin{eqnarray*}
 &&{\rm\frac{d}{dt}}\,J_{\rm mix,\Omega}(\bar{g}_t)_{\,|\,t=0}
 =\int_{\Omega}\Big\{\,\dt Q(\bar g_t)_{\,|\,t=0}
 +\frac12\,Q(g)\big(\tr_{g}{B} -(\tr_g B)(\Omega,g)\big)\Big\}\,{\rm d}\vol_g \\
 &&\hskip-7mm +\,{\rm\frac{d}{dt}}\int_{\Omega} \Div (H_{\bar g_t} + \tilde H_{\bar g_t})\,{\rm d}\vol_{\bar g_t}\,\!_{|\,t=0}
 =\int_{\Omega}\dt Q(g_t)_{\,|\,t=0}\,{\rm d}\vol_g +\frac12\int_{\Omega} Q(g)(\,\tr_{g} {B})\,{\rm d}\vol_g \\
 &&\hskip-7mm +\,\frac{1}{2}\Big(\frac2p\,(\Sm_{\rm\,ex} +2\<\tilde T,\tilde T\>_{g} -\<T,T\>_{g} +\Div H)
 -Q(g) -\Div(H +\tilde{H})\Big)({\Omega},g)\int_{\Omega}(\tr_{g} {B})\,{\rm d}\vol_g.
\end{eqnarray*}
Using definition of $Q(g)$ and
\eqref{E-S-star} we get \eqref{E-varJh-f}.
\qed

\begin{rem}\rm
It should be stressed that as in \cite{bdr}, we work with two types of variations of metric, (\ref{e:Var}) and (\ref{e:Var-Bar});
the second of which preserves the volume of ${\Omega}$.
Formulas containing $\Sm^*_{\rm mix}$ correspond to \eqref{e:Var-Bar}.
To obtain similar formulas, corresponding to $1$-parameter variations of the form (\ref{e:Var}),
one should merely delete the mean value terms $\Sm^*_{\rm mix}(\Omega,g)$ in the previous~identities.
Considering a closed manifold $M$ instead of $\Omega$, from the Divergence Theorem we have
\begin{equation*}
 \Sm^*_{\rm mix} = \Sm_{\rm mix} -
  \frac2p\,\big(\Sm_{\,\rm ex} +2\,\<\tilde T,\tilde T\> -\<T,T\> \big) \quad
  {\rm (for}~g^\perp\mbox{\rm-variations})\,.
\end{equation*}
\end{rem}


Next theorem gives the Euler-Lagrange equations of the variational principle $\delta J_{\rm mix,\Omega}(g)=0$
on a relatively compact domain $\Omega$ of a manifold $M$ with an almost-product structure.
These have a view $P=\lambda\,\tilde g$ (on $\widetilde{\cal D}$) and $P=\lambda\,g^\perp$ (on ${\cal D}$)
for certain tensors $P$ and functions $\lambda$.

\begin{thm}[\bf Euler-Lagrange equations]\label{T-main00}
Let
$g\in{\rm Riem}(M,\,\widetilde{\cal D},\,{\cal D})$ be a critical point of the action \eqref{E-Jmix}
with respect to
$g^\perp$-variations, \eqref{E-Sdtg}--\,\eqref{E-Sdtg-2}.
Then
\begin{eqnarray}\label{E-main-0i}
\nonumber
 &&\hskip-7mm {r}_{\cal D} -\<\tilde h,\,\tilde H\> +\widetilde{\cal A}^\flat -\widetilde{\cal T}^\flat
 +\Phi_h +\Phi_T +\Psi -{\rm Def}_{\cal D}\,H +\tilde{\cal K}^\flat \\
 &&=\frac12\,\big( \Sm_{\rm mix} - \Sm^*_{\rm mix}(\Omega,g) + \Div(\tilde H - H)\big)\,g^\perp.
\end{eqnarray}
\end{thm}

\begin{proof}
Applying Lemma~\ref{L-H2h2-D} to \eqref{E-Q-def}, using \eqref{E-divP} and removing integrals of divergences of vector fields compactly supported in $\Omega$, we get
\begin{equation*}
 \int_{\Omega} \dt Q_{\,|\,t=0}\,{\rm d}\vol_g = \int_{\Omega} \big\<
 \Div(\tilde H\,g^\perp - \tilde h) +2\,\widetilde{\cal T}^\flat -\Phi_h -\Phi_T -\tilde{\cal K}^\flat,\ {B}\big>\,{\rm d}\vol_g,
\end{equation*}
where ${B}=\{\dt g_t\}_{\,|\,t=0}\in{\mathfrak M}_{\cal D}$.
Notice that $\tr_{g} {B}=\<{B},\,g^\perp\>$. Then by \eqref{E-varJh-init} we have
\begin{eqnarray}\label{E-varJh-init2}
\nonumber
 {\rm\frac{d}{dt}}\,J_{\rm mix,\Omega}({g}_t)_{\,|\,t=0} \eq \int_{\Omega}
 \big\< \Div(\tilde H\,g^\perp - \tilde h) +2\,\widetilde{\cal T}^\flat -\Phi_h -\Phi_T -\tilde{\cal K}^\flat\\
 \plus\frac12\,\big(\Sm_{\rm mix} -\Div(H + \tilde H)\big)\,g^\perp,\ {B}\big>\,{\rm d}\vol_g\,.
\end{eqnarray}
By \eqref{E-varJh-init2} and Proposition~\ref{P-1-5} we obtain
\begin{eqnarray}\label{E-Jmix-dt-fin}
\nonumber
 {\rm\frac{d}{dt}}\,J_{\rm mix,\Omega}(\bar g_t)_{\,|\,t=0} \eq \!\int_{\Omega}\big\<
 \Div(\tilde H\,g^\perp - \tilde h) +2\,\widetilde{\cal T}^\flat-\Phi_h-\Phi_T -\tilde{\cal K}^\flat\\
 \plus\frac12\,\big(\Sm_{\rm mix} -\Sm^*_{\rm mix}(\Omega,g) -\Div(\tilde H +H) \big)\,g^\perp,\ {B}\big\>\,{\rm d}\vol_g.
\end{eqnarray}
If the metric $g$ is critical for the action $J_{\rm mix,\Omega}$ with respect to $g^\perp$-variations,
then the integral in \eqref{E-Jmix-dt-fin} is zero for arbitrary symmetric tensor ${B}\in{\mathfrak M}$ vanishing on $\widetilde{\cal D}$.
That yields
\begin{equation}\label{E-main-0ih-temp}
 \Div\tilde h -2\,\widetilde{\cal T}^\flat +\Phi_h +\Phi_T +\tilde{\cal K}^\flat
 =\frac12\,\big(\Sm_{\rm mix} -\Sm^*_{\rm mix}(\Omega,g)
 +\Div(\tilde H - H) \big)\,g^\perp.
\end{equation}
Using the partial Ricci tensor, see Proposition~\ref{L-CC-riccati},
and replacing ${\Div}\,\tilde h$ in \eqref{E-main-0ih-temp} by the value according to \eqref{E-genRicN}$_1$,
we rewrite \eqref{E-main-0ih-temp} as \eqref{E-main-0i}.
\end{proof}

\begin{rem}\rm
By~\eqref{E-main-0ih-temp} we conclude the following:
if a metric $g\in{\rm Riem(M,{\cal D},\widetilde{\cal D})}$ is critical with respect to $g^\perp$-variations
then the tensor $\Div\tilde h-2\,\widetilde{\cal T}^\flat+\Phi_{h}+\Phi_T+\tilde{\cal K}^\flat$ is ${\cal D}$-conformal.
\end{rem}

\begin{example}\label{Ex-Hopf}\rm
Let both distributions be totally geodesic. Then \eqref{E-main-0i} reads
\begin{eqnarray*}
  {r}_{\cal D} -\widetilde{\cal T}^\flat +\Phi_T +\Psi \eq
  \frac12\,\big( \Sm_{\rm mix} - \Sm^*_{\rm mix}(\Omega,g)\big)\,g^\perp,
\end{eqnarray*}
where
 $\Psi(X,Y) = \tr_{g}(T^\sharp_Y T^\sharp_X)\ (X,Y\in{\mathfrak X}_{\cal D})$.
 Also $\Sm_{\rm mix}=\<T,T\> + \<\tilde T,\tilde T\>$ and, see \eqref{E-S-star},
 \begin{equation*}
 \Sm^*_{\rm mix} = \frac1p\,\big((p-4)\,\<\tilde T,\tilde T\> +(p+2)\,\<T,T\> \big)
 \,.
\end{equation*}
Note that this is the case of Hopf fibrations, when $\widetilde{\cal D}$ is non-integrable, totally geodesic distribution with integrable orthogonal complement.
\end{example}

\subsection{Total extrinsic scalar curvature}

The variational formulas from Section \ref{sec:prel} can be applied also to other functionals depending on extrinsic geometry of distributions. In particular, we can consider integrals of extrinsic scalar curvatures $\widetilde{\Sm}_{\,\rm ex}$ and ${\Sm}_{\,\rm ex}$.
Since the variational formulas for both these quantities are similar,
we shall examine only $\widetilde{\Sm}_{\,\rm ex}$.
We consider adapted variations of the functional
\begin{equation}\label{Jwidetildeex}
 J_{\rm \widetilde{ex},\Omega}(g):\ g \rightarrow \int_{\Omega} \widetilde{\Sm}_{\,\rm ex}(g)\ {\rm d}\vol_g\,.
\end{equation}
Note that for $p=1$ we have $\widetilde{\Sm}_{\,\rm ex}=0$ for any metric.

\begin{prop}[\bf Euler-Lagrange equations]
Let $g\in{\rm Riem}(M,\,{T\calf},\,{\cal D})$ be critical for the action \eqref{Jwidetildeex}
with respect to adapted variations of $g$ and let $p>1$. Then
\begin{eqnarray}\label{CorEwidetildeex-main-0i}
&& \Div\tilde h +\tilde{\cal K}^\flat
 = -\frac{1}{2(p-1)}\,(\widetilde{\Sm}_{\,\rm ex} - \widetilde{\Sm}\,_{\rm ex}^{*}(\Omega, g)) \,g^{\perp},
 \quad ({\rm for}~g^\perp\mbox{\rm-variations}),\\
&& \Phi_{\tilde h} = \frac{1}{n}\,\widetilde{\Sm}_{\,\rm ex}\,{\tilde g},
 \ \ \mbox{and if } \ n\ne2 \  \mbox{then} \  \widetilde\Sm_{\,\rm ex}= \widetilde{\Sm}\,_{\rm ex}^{*}(\Omega, g)\quad
 ({\rm for}~\tilde g\mbox{\rm-variations})  , \label{CorEwidetildeex-main-0i-dual}
\end{eqnarray}
where $\widetilde{\Sm}\,_{\rm ex}^{*} = \widetilde{\Sm}_{\,\rm ex}$ for variations preserving the volume of $\Omega$ and $\widetilde{\Sm}\,_{\rm ex}^{*}=0$ otherwise.
\end{prop}

\begin{proof} 
The formula for $g^\perp$-variation of $\widetilde{\Sm}_{\,\rm ex}$ was given in \eqref{E-h2T2-D1},
and we can write the $\tilde g$-variation of $\widetilde{\Sm}_{\,\rm ex}$ from \eqref{E-h2-D1} as
\begin{equation}\label{dtwidetildeSm}
 \dt\,\widetilde{\Sm}_{\,\rm ex} = -\<\Phi_{\widetilde{h}},\,{B}\>,
\end{equation}
interchanging the roles of $\cal D$ and $\cal \widetilde{D}$.
 Using \eqref{E-h2T2-D1}, \eqref{E-dtdvol}, \eqref{dtwidetildeSm}, and removing divergences of compactly supported vector fields, we obtain for $g^\perp$-variations:
\begin{eqnarray*}
 {\rm\frac{d}{dt}}\,J_{\rm \widetilde{ex},\Omega}(g_t)_{\,|\,t=0}
 \eq \int_{\Omega} (\dt\widetilde{\Sm}_{\,\rm ex}) \ {\rm d}\vol_g
 +\int_{\Omega} \widetilde{\Sm}_{\,\rm ex}\,\dt({\rm d}\vol_g)\, \\
 \eq  \int_{\Omega} \<\,(\Div\tilde H)\,g^\perp -\Div\tilde h -\tilde{\cal K}^\flat,\, {B}\>\ {\rm d}\vol_g\,
 + \frac{1}{2} \int_{\Omega} \widetilde{\Sm}_{\,\rm ex} (\tr_{g_t} {B}_t) \,{\rm d}\vol_{g_t} \\
 \eq \int_{\Omega} \<\,(\Div\tilde H)\,g^\perp -\Div\tilde h -\tilde{\cal K}^\flat
 + \frac{1}{2}\,\widetilde{\Sm}_{\,\rm ex}\,g^{\perp} , \, {B}\>\,{\rm d}\vol_g,
\end{eqnarray*}
and for $\tilde g$-variations:
\begin{equation*}
 {\rm\frac{d}{dt}}\,J_{\widetilde{\rm ex},\Omega}(g_t)_{\,|\,t=0}
 =\int_{\Omega} \<-\Phi_{\tilde h} + \frac{1}{2}\,\widetilde{\Sm}_{\,\rm ex}\,{\tilde g},\, {B}\> \ {\rm d}\vol_g .
\end{equation*}
In the case of variations preserving the volume of
$\Omega$, using the notation of Section~\ref{subsec:EU-general} and methods employed in the proof of Proposition \ref{P-1-5}, we get for $g^\perp$-variations
\[
 {\rm\frac{d}{dt}}\,J_{\rm \widetilde{ex},\Omega}(\bar g_t)_{\,|\,t=0} ={\rm\frac{d}{dt}}\,J_{\rm \widetilde{ex},\Omega}(g_t)_{\,|\,t=0}
 -\frac{1}{2}\,\widetilde{\Sm}_{\,\rm ex}(\Omega, g) \int_{\Omega} \big(\tr_{g} {B}\big)\,{\rm d}\vol_g,
\]
and for $\tilde g\,$-variations:
\[
 {\rm\frac{d}{dt}}\,J_{\rm \widetilde{ex},\Omega}(\bar g_t)_{\,|\,t=0} = {\rm\frac{d}{dt}}\,J_{\rm \widetilde{ex},\Omega}(g_t)_{\,|\,t=0}
- \frac{n-2}{2n}\,\widetilde{\Sm}_{\,\rm ex}(\Omega, g) \int_{\Omega} \big(\tr_{g} {B}\big)\,{\rm d}\vol_g .
\]
Therefore, we obtain the following Euler-Lagrange equations for the action \eqref{Jwidetildeex}
(terms $\widetilde{\Sm}\,^{*}_{\rm ex}$ appear only in case of variations preserving the volume of $\Omega$):
\begin{eqnarray}\label{Ewidetildeex-main-0i}
 &&\hskip-7mm
 \Div\tilde h +\tilde{\cal K}^\flat
 = \big(\Div\tilde H + \frac{1}{2}\,( \widetilde{\Sm}_{\,\rm ex} - \widetilde{\Sm}\,_{\rm ex}^{*}(\Omega, g))\big)\,g^{\perp}
 \quad ({\rm for}~g^\perp\mbox{\rm-variations}),\\
 \label{Ewidetildeex-main-0i-dual}
 &&\hskip-7mm  \Phi_{\tilde h}
 =\frac{1}{2}\,\big(\widetilde{\Sm}_{\,\rm ex} -\frac{n-2}{n}\,\widetilde{\Sm}\,_{\rm ex}^{*}(\Omega, g)\big)\,{\tilde g}
  \quad ({\rm for}~\tilde g\mbox{\rm-variations}).
\end{eqnarray}
Taking traces of \eqref{Ewidetildeex-main-0i} and \eqref{Ewidetildeex-main-0i-dual} yields
\begin{equation}\label{trEwidetildeex-main-0i}
 \Div \tilde H = \frac{p}{2\,(1-p)}\,( \widetilde{\Sm}_{\,\rm ex} - \widetilde{\Sm}\,_{\rm ex}^{*}(\Omega, g)), \quad
 (n-2)\,( \frac{1}{2}\,\widetilde{\Sm}_{\,\rm ex} -  \widetilde{\Sm}\,_{\rm ex}^{*}(\Omega, g)) = 0.
\end{equation}
Using \eqref{trEwidetildeex-main-0i}$_1$ in \eqref{Ewidetildeex-main-0i} and \eqref{trEwidetildeex-main-0i}$_2$ in \eqref{Ewidetildeex-main-0i-dual} completes the proof.
\end{proof}


\section{Applications}
\label{sec:main}

In this section we assume that a pseudo-Riemannian manifold $(M, g)$ is endowed with an $n$-dimensi\-onal foliation $\calf$.
Since $\widetilde{\cal D}=T\calf$, we write ${r}_{\,\calf}={r}_{\widetilde{\cal D}}$,
and obtain dual to \eqref{E-genRicN} equations
\begin{eqnarray}\label{E-genRicN-2}
 {r}_{\,\calf} = \Div h +\<h,\,H\> -{\cal A}^\flat -\widetilde\Psi +{\rm Def}_{\calf}\,\tilde H,\qquad
 d_\calf\,\tilde H=0\,.
\end{eqnarray}
Note that $\Psi(X,Y)=\tr_{g}(A_Y A_X)$.
Definition \eqref{E-S-star} takes the form
\begin{equation}\label{E-S*main}
 \Sm^*_{\rm mix} = \Sm_{\rm mix} - \left\{\begin{array}{cc}
  \frac2p\,\big(\Sm_{\,\rm ex} +2\,\<\tilde T,\tilde T\>+\Div H\big) & ({\rm for}~g^\perp\mbox{\rm-variations}),\\
  \frac2n\,\big(\widetilde\Sm_{\,\rm ex} -\<\tilde T,\tilde T\>+\Div\tilde H\big) &
  ({\rm for}~\tilde g\mbox{\rm-variations})\,.
  \end{array}\right.
\end{equation}

From Theorem~\ref{T-main00}, we obtain the following.

\begin{cor}[\bf Euler-Lagrange equations]\label{T-main01}
Let $\calf$ be a foliation with a transversal distribution ${\cal D}$ on~$M$,
and $g\in{\rm Riem}(M,\,{T\calf},\,{\cal D})$ be critical for the action \eqref{E-Jmix}
with respect~to adapted variations,~\eqref{E-Sdtg}--\,\eqref{E-Sdtg-2}.
Then
\begin{eqnarray}\label{E-main-i}
\nonumber
 &&\hskip-7mm {r}_{{\cal D}} -\<\tilde h,\,\tilde H\> +\widetilde{\cal A}^\flat
 -\widetilde{\cal T}^\flat +\Phi_h +\Psi -{\rm Def}_{\cal D}\,H +\tilde{\cal K}^\flat\\
 && = \frac12\,\big(\Sm_{\rm mix} -\Sm^*_{\rm mix}(\Omega,g) +\Div(\tilde H - H) \big)\,g^\perp
 \quad({\rm for}~g^\perp\mbox{\rm-variations}), \\
\label{E-main-ii}
 \nonumber
 &&\hskip-7mm {r}_{\,\calf} -\<h,\,H\> +{\cal A}^\flat
 +\Phi_{\tilde h} +\Phi_{\tilde T} +\widetilde\Psi -{\rm Def}_{\calf}\,\tilde H \\
 && =\frac12\,\big( \Sm_{\rm mix} -\Sm^*_{\rm mix}(\Omega,g) +\Div(H-\tilde H) \big)\,\tilde g
 \quad({\rm for}~\tilde g\mbox{\rm-variations})\,.
\end{eqnarray}
\end{cor}

These mixed field equations admit a certain number of solutions (e.g., twisted products, see below),
we propose that they will find applications in theoretical physics, see discussion in~\cite{bdrs}.

A pseudo-Riemannian manifold may admit many different  geometrically interesting types of foliations: totally geodesic ($h=0$) and Riemannian ($\tilde h=0$) foliations are the  most common examples;
totally umbilical ($h=\frac1nH\,\tilde g$)  and conformal ($\tilde h=\frac1p\tilde H\,g^\bot$) foliations are also popular.
The simple examples of geodesic foliations are
parallel circles or winding lines on a flat torus.

\begin{example}\rm
 Let $\cal F$ be a totally umbilical foliation (i.e., $h=\frac1n\,H\tilde g$ and $T=0$). Then
\[
 \Phi_h=\frac{n-1}n\,{H}^\flat\otimes{H}^\flat,\quad
 {\cal A}^\flat = \frac{1}{n^2}\,g(H,H)\,\tilde g,\quad
 \Psi=\frac1n\,{H}^\flat\otimes {H}^\flat,\quad
 \Sm_{\,\rm ex} = \frac{n-1}{n}\,g(H,H).
\]
Hence, the fundamental equation \eqref{E-genRicN}$_1$ and the Euler-Lagrange equation \eqref{E-main-i} read as
\begin{eqnarray}\label{E-umbRD-b}
 && \hskip-5mm r_{\,\cal D} -\Div\tilde h -\<\tilde h,\,\tilde H\> +\widetilde{\cal A}^\flat
 +\widetilde{\cal T}^\flat +\frac1n\,{H}^\flat\otimes {H}^\flat -{\rm Def}_{\cal D}\,H = 0,\\
\label{E-main-TUmb}
\nonumber
 &&\hskip-5mm {r}_{{\cal D}} -\<\tilde h,\,\tilde H\> +\widetilde{\cal A}^\flat
 -\widetilde{\cal T}^\flat +{H}^\flat\otimes {H}^\flat -{\rm Def}_{\cal D}\,H +\tilde{\cal K}^\flat \\
 && = \frac12\,\big(\Sm_{\rm mix} -\Sm^*_{\rm mix}(\Omega,g) +\Div(\tilde H - H) \big)\,g^\perp
 \quad({\rm for}~g^\perp\mbox{\rm-variations}).
\end{eqnarray}
\end{example}

\subsection{Critical adapted metrics}
\label{subsec:EL-fol}

In this section we examine sufficient conditions for a metric $g$ to be critical with respect to $g^\perp$-variations.
The~case of $\tilde g$-variations is similar.

\begin{thm}\label{C-2umb}
Let a metric $g\in{\rm Riem}(M,\,{T\calf},\,{\cal D})$ be critical for the action \eqref{E-Jmix} with respect
to $g^\perp$-variations,  $n,p>1$, and ${\cal D}$ and $\widetilde{\cal D}$ determine totally umbilical foliations.
Then the leaves of $\widetilde{\cal D}$ are totally geodesic~and
\begin{equation}\label{E-DTwist}
 {r}_{{\cal D}} = (\Sm_{\rm mix}/p)\,g^\bot\quad {\rm and}\quad
 \left\{\begin{array}{cc}
  \int_\Omega(\Div\tilde H)\,{\rm d}\vol=0 & {\rm if}~p\ne2,\\
  \Sm_{\rm mix} = \const & {\rm if}~p=2\,.
  \end{array}\right.
\end{equation}
\end{thm}

\begin{proof}
We have the identity, see \eqref{E-umbRD-b} with $\tilde T=0$ and $\tilde h=\frac1p\,\tilde H\,g^\perp$,
\begin{equation}\label{E-dtp-gen}
 r_{{\cal D}} +\frac1n\,H^\flat\otimes H^\flat -{\rm Def}_{\cal D}\,H
 = \frac1p\,\Big(\frac{p-1}{p}\,g(\tilde H,\tilde H) +\Div\tilde H \Big) g^\perp\,.
\end{equation}
Hence, or by \eqref{eq-ran-ex},
\[
 \Sm_{\rm mix} =\frac{n-1}{n}\,g(H,H) +\frac{p-1}{p}\,g(\tilde H,\tilde H) +\Div(H+\tilde H).
\]
Let the metric $g$ be critical with respect to $g^\perp$-variations. By \eqref{E-main-TUmb} we have
\begin{equation}\label{E-dtw-1A}
 r_{{\cal D}} + H^\flat\otimes H^\flat -{\rm Def}_{\cal D}\,H
  = \frac12\,\Big(\Sm_{\rm mix} -\Sm^*_{\rm mix}(\Omega,g) + \frac{2(p-1)}{p^2}\,g(\tilde H,\tilde H)
  +\Div(\tilde H - H)\Big)\,g^\perp.
\end{equation}
The difference of \eqref{E-dtw-1A} and \eqref{E-dtp-gen} is
\begin{equation*}
 \frac{n-1}{n}\,H^\flat\otimes H^\flat
 =\frac12\,\Big(\frac{n-1}{n}\,g(H,H) +\frac{p-1}{p}\,g(\tilde H,\tilde H) -\Sm^*_{\rm mix}(\Omega,g)
 +\frac{2(p-1)}{p}\,\Div\tilde H\Big)\,g^\perp.
\end{equation*}
As the symmetric $(0,2)$-tensor $H^\flat\otimes H^\flat$ has rank $\le 1$ and $g^\perp$ has rank $p>1$,
we obtain $H=0$; hence, the leaves of $\widetilde{\cal D}$ are totally geodesic.
By \eqref{E-dtp-gen}, the
tensor $r_{{\cal D}}$ is ${\cal D}$-conformal.
We~also~have
$\Sm_{\rm mix} + \frac{p-2}{p}\Div\tilde H =\Sm^*_{\rm mix}(\Omega,g)$, where $\Sm^*_{\rm mix}=\Sm_{\rm mix}$.
Thus,
$\int_\Omega(\Div\tilde H)\,{\rm d}\vol=0$ for $p\ne2$.
\end{proof}

\begin{example}\label{C-dtwist}\rm
Let $M = M_1\times M_2$ be the product of pseudo-Riemannian manifolds $(M_i\,,\,g_i)$ ($i\in\{1,2\}$),
and let $\pi_i : M\to M_i$ and $d\pi_i: TM\to TM_i$ be canonical projections.
Given twisting functions $f_i\in C^\infty(M)$,
a \textit{double-twisted product} $M_1\times_{(f_1,f_2)} M_2$ is $M$ with the metric
 $g =e^{f_1} \, \pi_1^\ast g_1 + e^{f_2} \, \pi_2^\ast g_2$.
If $f_1 = \const$ then we have a {twisted product}
(a warped product if, in addition, $f_2 = F \circ \pi_1$ for some $F \in C^\infty (M_1)$).
The {leaves} $M_1\times\{y\}$ (tangent to $\widetilde{\cal D}$) and the {fibers} $\{x\}\times M_2$
(tangent to ${\cal D}$)
are totally umbilical in $(M,g)$ and this property characterizes double-twisted products  (cf. \cite{pr}).
For any double-twisted product, we~have $T=0$~and
\begin{eqnarray*}
 A_Y\eq -Y(f_1)\,\widetilde\id,\quad \ \ h=-(\nabla^\perp f_1)\,\tilde g,\quad H=-n\nabla^\perp f_1,
\end{eqnarray*}
(and similarly for $\tilde T,\tilde A_X,\tilde h,\tilde H$)
where $X\in\widetilde{\cal D}$ and $Y\in{\cal D}$ are unit vectors.
In this case, see \eqref{E-divN},
\begin{equation*}
 \Div\tilde H= -p\,\widetilde\Delta\,f_2 -p^2 g(\widetilde\nabla f_2,\widetilde\nabla f_2),\quad
 \Div H= -n\,\Delta^\perp f_1 -n^2 g(\nabla^\perp f_1,\nabla^\perp f_1).
\end{equation*}
The $\widetilde{\cal D}$-{Laplacian} of a function $f$ is given by the formula
$\widetilde\Delta\,f=\widetilde{\Div}\,(\widetilde\nabla\,f)$.
By \eqref{eq-ran-ex},
\[
 \Sm_{\rm mix} =\Div(H+\tilde H)+\frac{n-1}{n}\,g(H,H) +\frac{p-1}{p}\,g(\tilde H,\tilde H).
\]
 Let $g$ be critical for the action \eqref{E-Jmix} with respect to $g^\perp$-variations.
By Theorem~\ref{C-2umb}, $H=0$, the leaves are totally geodesic, and \eqref{E-DTwist} hold.
Note that
 $p\,\widetilde\Delta\,f_2 +p^2 g(\widetilde\nabla f_2,\widetilde\nabla f_2)=0$
is equivalent to the equality $\widetilde\Delta\,e^{p\,f_2}=0$.

 Summarizing, we conclude that a pseudo-Riemannian double-twisted product metric $g$
is critical for the action \eqref{E-Jmix} with respect to $g^\perp$-variations if the following conditions hold:

\smallskip
(i)~$r_{\cal D}$ is ${\cal D}$-conformal; \quad

(ii) if $p\ne2$ then $\widetilde\Delta\,e^{p\,f_2}=0$
 (hence, $e^{p\,f_2}$ is $\widetilde{\cal D}$-harmonic when $g_{\,|\widetilde{\cal D}}$ is definite);

(iii)~$f_1$ does not depend on $M_2$, i.e., the twisted product of $(M_1\,,\,e^{f_1}g_1)$ and $(M_2\,,\,g_2)$.

\smallskip
\noindent
Nonconstant bounded (e.g. positive) harmonic functions exist
on a complete mani\-fold with nonnegative curvature outside a compact set \cite{LT}.
By
S.T.\,Yau theorem (1975), there are no
non\-constant positive harmonic functions on a complete manifold with nonnegative Ricci curvature.
\end{example}

The following theorem continues Example~\ref{Ex-Hopf}: one of distributions becomes integrable.

\begin{thm}\label{T-TG-Rfol}
Let a metric $g\in{\rm Riem}(M,\,{T\calf},\,{\cal D})$ be critical for the action \eqref{E-Jmix} with respect to $g^\perp$-variations, ${\cal D}$ nowhere integrable $($hence, $p>1)$ and $\widetilde{\cal D}$ tangent to a~totally geodesic Riemannian foliation. Then
\begin{equation}\label{E-umb-Riem}
 {r}_{\cal D} = (\Sm_{\rm mix}/p)\,g^\bot,\quad
 {\rm where}\ \Sm_{\rm mix}={\rm const}\ \ {\rm when}\ \ p\ne4.
\end{equation}
\end{thm}

\proof
By conditions, $h=0=\tilde h$ and $T=0$. Thus, \eqref{E-genRicN}$_1$ reads as
\begin{eqnarray}\label{E-geodRD-b}
 r_{\,\cal D} \eq -\widetilde{\cal T}^\flat.
\end{eqnarray}
Tracing \eqref{E-geodRD-b}, we find $\Sm_{\rm mix}=\<\tilde T,\tilde T\>$.
From \eqref{E-main-i} we obtain
\begin{eqnarray}\label{E-main-i-fol}
 {r}_{{\cal D}} -\widetilde{\cal T}^\flat  \eq \frac12\,\big(\Sm_{\rm mix} -\Sm^*_{\rm mix}(\Omega,g)\big)\,g^\perp
  \quad({\rm for}~g^\perp\mbox{\rm-variations}),
\end{eqnarray}
where $\Sm^*_{\rm mix}=\frac{p-4}p\,\<\tilde T,\tilde T\>$, see \eqref{E-S*main}.
Adding \eqref{E-geodRD-b} and \eqref{E-main-i-fol}, we obtain
\begin{equation}\label{E-rD-Sm-1}
 {r}_{{\cal D}} = \frac14\,\big(\Sm_{\rm mix}
 -\Sm^*_{\rm mix}(\Omega,g)\big)\,g^\perp.
\end{equation}
Tracing \eqref{E-rD-Sm-1}, we get $(p-4)\Sm_{\rm mix} =p\,\Sm^*_{\rm mix}(\Omega,g)$,
hence, $\Sm_{\rm mix}={\rm const}$ when $p\ne4$.
This and \eqref{E-rD-Sm-1} complete the proof.
\qed

\begin{thm}\label{T-TG-int}
Let $\calf$ be a totally geodesic foliation of a pseudo-Riemannian manifold $(M,g)$ with integrable normal bundle ${\cal D}$.
 If $g$ is critical for the action \eqref{E-Jmix} with respect to adapted variations
then the following conditions hold:
\begin{equation}\label{E-scalJ-tgeo-1}
 {\rm (i)}~\Div\big(\tilde h - \frac1p\,\tilde H\,g^\perp\big) = 0,\quad
 {\rm (ii)}~\Phi_{\tilde h}=\frac1n\,\widetilde\Sm_{\,\rm ex}\,\tilde g,\ \
{\rm and}\ \Sm_{\rm mix}={\rm const}\ \ {\rm when}\ \ n\ne2\,.
\end{equation}
\end{thm}

\begin{proof}
 Using \eqref{E-main-0ih-temp} and its dual with $\tilde T=0$,
 rewrite Euler-Lagrange equations \eqref{E-main-i}\,--\,\eqref{E-main-ii} as
\begin{eqnarray}\label{E-main-i-hF}
 \Div(\tilde h -\tilde H\,g^\perp) +\Phi_h = \frac12\,\big(\Sm_{\,\rm ex} +\widetilde\Sm_{\,\rm ex}
  -\Sm^*_{\,\rm mix}(\Omega,g)\big)\,g^\perp \ \ ({\rm for}~g^\perp\mbox{\rm-variations}), \\
\label{E-main-ii-hF}
 \Div(h -H\,\tilde g) +\Phi_{\tilde h} = \frac12\,\big(\Sm_{\,\rm ex} +\widetilde\Sm_{\,\rm ex}
  -\Sm^*_{\,\rm mix}(\Omega,g)\big)\,\tilde g \quad({\rm for}~\tilde g\mbox{\rm-variations}),
\end{eqnarray}
We need to show the following (for totally geodesic foliations with integrable normal bundle):

\smallskip\noindent\
{\rm (i)}~if $g$ is critical for the action $J_{\rm mix,\Omega}$ with respect to $g^\perp$-variations
then \eqref{E-scalJ-tgeo-1}(i) holds; and

\smallskip\noindent\
{\rm (ii)} if $g$ is critical for $J_{\rm mix,\Omega}$ with respect to $\tilde g$-variations
then \eqref{E-scalJ-tgeo-1}(ii) holds.

\smallskip\noindent
To show (i), observe that for totally geodesic foliations $h=0$; hence, (\ref{E-main-i-hF}) reads:
\begin{equation}\label{divtildeh}
 \Div\,(\tilde h - \tilde H g^{\perp})
 = -\frac{1}{2}\,\big(\,\widetilde\Sm_{\,\rm ex} + \Sm^*_{\,\rm mix}(\Omega,g)\,\big) \,g^\perp\,,
\end{equation}
where $\Sm^*_{\,\rm mix}=\Sm_{\,\rm mix}$. Taking trace of \eqref{divtildeh} yields
\begin{equation}\label{trdivtildeh}
 (1-p) \Div \tilde H  = \frac{p}{2}\,\big(\,\widetilde\Sm_{\,\rm ex} - \Sm^*_{\,\rm mix}(\Omega,g)\,\big).
\end{equation}
Therefore, introducing the values of \eqref{trdivtildeh} into \eqref{divtildeh}, we obtain
\begin{equation*}
 \Div\big(\tilde h - \frac1p\,\tilde H\,g^\perp\big) = \Div\big(\tilde h - \tilde H\,g^\perp\big)
  +\frac{p-1}{p}\,(\Div \tilde H) \,g^\perp
  = \big( \frac{1-p}{p} + \frac{p-1}{p} \big)\Div \tilde H\,g^\perp = 0.
\end{equation*}
 To show (ii), from (\ref{E-main-ii-hF}) with $h=0$ we obtain for $T\calf$-variations,
\begin{equation}\label{Phitilde}
 \Phi_{\tilde h} = \frac12\,\big( \widetilde\Sm_{\,\rm ex} -\Sm^*_{\,\rm mix}(\Omega,g)\big)\,\tilde g\,,
\end{equation}
where $\Sm^*_{\,\rm mix}=\Sm_{\,\rm mix}-\frac2n(\widetilde{\Sm}_{\,\rm ex}+\Div\tilde H)$.
Tracing \eqref{Phitilde} yields
\begin{equation}\label{trPhitilde}
 \widetilde\Sm_{\,\rm ex} = \frac{n}{2}\,\big(\widetilde\Sm_{\,\rm ex} - \Sm^*_{\,\rm mix}(\Omega,g)\big),
\end{equation}
then introducing the values of \eqref{trPhitilde} into \eqref{Phitilde} we obtain
 $\Phi_{\tilde h} = \frac{1}{n}\,\widetilde\Sm_{\,\rm ex} \, \tilde g$.
It also follows from (\ref{trPhitilde}) that
 $\frac{2-n}{n}\,\widetilde\Sm_{\,\rm ex} = - \frac{n}{2}\,\Sm^*_{\,\rm mix}(\Omega,g)$
 for $n \neq 2$, while for $n=2$ we have $\Sm^*_{\,\rm mix}(\Omega,g) =0$.
\end{proof}

In light of Theorems~\ref{T-TG-Rfol} and \ref{T-TG-int}, it might be interesting to
study totally geodesic foliations


\smallskip
$(a)$ {with totally geodesic normal bundle and for which
\eqref{E-umb-Riem} holds}.

\smallskip
$(b)$ {with integrable normal bundle and for which conditions \eqref{E-scalJ-tgeo-1} hold}.


\subsection{Flows ($n=1$)}\label{subsec:dim1fol}

Let $\widetilde{\cal D}$ be spanned by a nonsingular vector field $N$, then $N$ defines a flow
(a one-dimensional foliation). An example is provided by a circle action $S^1 \times M \to M$ without fixed points.
 Assume that $| g(N,N) | = 1$ and denote $\epsilon_{N} = g(N,N)$.
Thus, $\Sm_{\rm mix}=\epsilon_N\Ric_{N}$, and
the partial Ricci tensor takes a particularly simple form:
\[
 r_{\widetilde{\cal D}}=\epsilon_{N} \Ric_{N}\,\tilde g,\qquad
 r_{\,\cal D}=\epsilon_{N} (R_N)^\flat,
\]
where $R_N\!=R(N,\,\cdot\,)N$ and
$\Ric_{N}=\sum_i \eps_i\,g(R_N({\cal E}_i),{\cal E}_i)$.
The action \eqref{E-Jmix} reduces itself to
\begin{equation}\label{E-Jmix-N}
 J_{\rm mix,\Omega}(g) = \epsilon_{N}\int_{\Omega} \Ric_{N}\,{\rm d}\vol_g.
\end{equation}
We have $\tilde h = \tilde h_{sc} N$,
where $\tilde h_{sc} = \epsilon_{N} \< \tilde h , N \>$ is the scalar second fundamental form of~${\cal D}$.

Define the functions $\tilde\tau_i=\tr\tilde A_N^{\,i}\ (i\ge0)$.
 It is easy to check that $\widetilde{\Sm}_{\,\rm ex}=\tilde\tau_1^2 -\tilde\tau_2$ and
\begin{eqnarray*}
 \Div N \eq \sum\nolimits_{\,i} \epsilon_{i}\,g(\nabla_i N, {\cal E}_i)
 = -g(N,\sum\nolimits_{\,i}\! \epsilon_{i}\nabla_i\,{\cal E}_i)
 = -g(N, \tilde H) = -\tilde\tau_1,\\
 \Div(\tilde\tau_1 N) \eq N(\tilde\tau_1) +\tilde\tau_1\Div N=N(\tilde\tau_1)-\tilde\tau_1^2.
\end{eqnarray*}
The curvature of the flow lines is $H=\epsilon_N\,\nabla_{N}\,N$.
 It is easy to see that \eqref{E-S*main} takes the form
\begin{equation*}
 \Sm^*_{\rm mix} = \epsilon_{N}\Ric_{N} -\,2\left\{\begin{array}{cc}
 \frac2p\,\<\tilde T,\tilde T\> +\frac1p\,\Div H & ({\rm for}~g^\perp\mbox{\rm-variations}),\\
 \epsilon_{N}( N(\tilde\tau_1)-\tilde\tau_2)-\<\tilde T,\tilde T\> &
 ({\rm for}~\tilde g\mbox{\rm-variations})\,.
\end{array}\right.
\end{equation*}

From Theorem~\ref{T-main00} (of Corollary~\ref{T-main01}) we obtain the following.

\begin{cor}[\bf Euler-Lagrange equations]
Let the distribution $\widetilde{\cal D}$ be spanned by a non-singular vector field~$N$, and
a pseudo-Riemannian metric $g\in{\rm Riem}(M,\,\widetilde{\cal D},\,{\cal D})$ be critical for the action
\eqref{E-Jmix-N} with respect to adapted variations. Then
\begin{eqnarray}\label{E-main-1i}
\nonumber
 && \epsilon_{N} \big( R_N +\tilde A_N^2 -(\tilde T^\sharp_N)^2 +[\tilde T_N^\sharp,\tilde A_N]\big)^\flat \!-\tilde\tau_1\tilde h_{sc}
 +H^\flat\otimes H^\flat -{\rm Def}_{\cal D}\,H  \\
 &&\hskip4mm =\frac12\,\big(\epsilon_{N} \Ric_{N} -\Sm^*_{\rm mix}(\Omega,g)
 +\Div(\epsilon_{N} \tilde\tau_1 N - H)\big)\,g^\perp \ \ ({\rm for}~g^\perp\mbox{\rm-variations}),\\
\label{E-main-2i}
 && \epsilon_{N} \Ric_{N} +\,\Sm^*_{\rm mix}(\Omega,g) -4 \< \tilde T , \tilde T \>  -\Div(\epsilon_{N}\tilde\tau_1 N + H) =0
 \ \ ({\rm for}~\tilde g\mbox{\rm-variations}).
\end{eqnarray}
\end{cor}

\begin{proof}
 An easy computation shows that
\begin{eqnarray}\label{E-Jmix-dim1}
\nonumber
 \widetilde{\cal A} \eq \epsilon_{N} \tilde A_N^2,\quad
  \<\tilde h_{sc} N,\,\tilde H\> = \tilde\tau_1\tilde h_{sc},\quad \Psi = H^\flat\otimes H^\flat,\quad
  \widetilde\Psi = (\epsilon_{N} \tilde\tau_2 - \< \tilde T , \tilde T \> )\,\tilde g,\\
\nonumber
 {\cal A} \eq g(H,H) \,\widetilde\id,\quad
 {\cal T}  = 0,\quad
 \<h,\,H\> = g(H,H) \tilde g ,\\
\nonumber
 H\eq \epsilon_{N} \nabla_N\,N,\quad h=H\,\tilde g,\quad \<h,h\>=g(H,H) , \\
 \tilde H\eq \epsilon_{N} \tilde\tau_1 N,\quad \tilde\tau_1= \epsilon_{N} \tr_{g}\tilde h_{sc},\quad
 \< \tilde h , \tilde h \> = \epsilon_{N} \tilde\tau_2,\quad
  {\rm Def}_{\widetilde{\cal D}}\,\tilde H = \epsilon_{N} N(\tilde\tau_1)\,\tilde g\,.
\end{eqnarray}
 Notice that $(H^\flat\otimes H^\flat)(X,Y)=g(H,X)\,g(H,Y)$.
Introducing the values \eqref{E-Jmix-dim1} and
\[
 \Phi_h=0=\Sm_{\,\rm ex},\quad
 \widetilde{\Sm}_{\,\rm ex}=\epsilon_{N} ( \tilde\tau_1^2-\tilde\tau_2),\quad
 \widetilde{\cal T}=\epsilon_{N} \tilde T^{\sharp\,2}_N
\]
into \eqref{E-main-0i} yields~\eqref{E-main-1i}. Introducing the values \eqref{E-Jmix-dim1} and
\[
 h =H\,\tilde g,\quad
 \Phi_{\tilde h} = \epsilon_{N} (\tilde\tau_1^2-\tilde\tau_2)\,\tilde g,\quad
 \Phi_{\tilde T} = -\< \tilde T , \tilde T \>\,\tilde g
\]
into equation dual to \eqref{E-main-0i} yields \eqref{E-main-2i}.
\end{proof}

By \eqref{E-divP}, we have $\Div \tilde h = N(\tilde h_{sc}) - \tilde\tau_1\tilde h_{sc}$ and $\Div h =(\Div H)\,\tilde g$.
 Then, see \eqref{E-genRicN}$_1$ and \eqref{eq-ran-ex},
\begin{eqnarray}\label{E-RicNs1aa}
\nonumber
 && \epsilon_{N} \big(R_N + \tilde A_N^2+(\tilde T^\sharp_N)^2\big)^\flat
 = N(\tilde h_{sc}) -H^\flat\otimes H^\flat +{\rm Def}_{\cal D}\,H,\\
 && \epsilon_{N}\Ric_{N}
 = \epsilon_{N}\Div(\nabla_N\,N) + \epsilon_{N} (N(\tilde\tau_1) -\tilde\tau_2) + \< \tilde T, \tilde T \>.
\end{eqnarray}
Remark that \eqref{E-RicNs1aa}$_2$ is simply the trace of \eqref{E-RicNs1aa}$_1$.

 A flow of a unit vector
 $N$
 is \textit{geodesic} if the orbits are geodesics ($h=0$),
 and
is \textit{Riemannian} if the metric is bundle-like ($\tilde h=0$).
A nonsingular Killing vector clearly defines a Riemannian flow;
moreover, a Killing vector of unit length generates a geodesic Riemannian~flow.
A manifold with such $N$-flow is called \textit{Sasakian} if
the sectional curvature of every section containing $N$ equals one,
in other words, its curvature satisfies the following condition:
\[
 R(X,N)Y=g(N,Y)X-g(X,Y)N\,.
\]

\begin{cor}[of Theorem~\ref{T-TG-Rfol}]
Let a unit vector field $N$ generates a geodesic Riemannian flow on a pseudo-Riemannian manifold $(M^{p+1},g)$.
If $g$ is critical for the action \eqref{E-Jmix-N} with respect to $g^\perp$-variations then
\begin{equation}\label{E-1geod-Riem}
 R_N = (1/p)\,\Ric_{N}\id^\perp,\quad {\rm where}\ \Ric_{N}={\rm const}\ \ {\rm when}\ \ p\ne4.
\end{equation}
Moreover,
if~$p$ is odd then $K_{\rm mix}=0$ and $M$ splits, and if $K_{\rm mix}\ne0$ then $p$ is even
and for $p\ne4$ $K_{\rm mix}$ is a function of a point only.
\end{cor}

\begin{proof}
By Theorem~\ref{T-TG-Rfol}, we have \eqref{E-1geod-Riem}, and \eqref{E-geodRD-b} reads $R_N =-(\tilde T^\sharp_N)^2$.
Tracing this we obtain $\epsilon_N\Ric_{N}=\<\tilde T,\tilde T\>$.
In our case, \eqref{E-S-star} reads
\begin{equation*}
 \Sm^*_{\rm mix} = \Big\{\begin{array}{cc}
 \frac{p-4}p\,\<\tilde T,\tilde T\> & ({\rm for}~g^\perp\mbox{\rm-variations}),\\
 3\,\<\tilde T,\tilde T\>  & ({\rm for}~\tilde g\mbox{\rm-variations})\,.
 \end{array}
\end{equation*}
For a geodesic Riemannian $N$-flow, \eqref{E-main-1i}--\eqref{E-main-2i} reduce to
\begin{eqnarray*}
 \epsilon_{N} \big( R_N -(\tilde T^\sharp_N)^2 \big)^\flat \eq \frac12\,\big( \epsilon_{N} \Ric_{N} -\Sm^*_{\rm mix}(\Omega,g)\big)\,g^\perp
 \ \ ({\rm for}~g^\perp\mbox{\rm-variations}),\\
 \epsilon_{N} \Ric_{N} \eq -\Sm^*_{\rm mix}(\Omega,g) +4\,\< \tilde T , \tilde T \>
 \qquad ({\rm for}~\tilde g\mbox{\rm-variations}).
\end{eqnarray*}
For $p$ odd, the skew-symmetric operator $\tilde T^\sharp_N$ has zero eigenvalue; hence, $R_N=0=\tilde T$;
and by de Rham Decomposition Theorem, $(M,g)$ splits.
\end{proof}

\smallskip

Finally, remark that we can examine codimension-one foliations and distributions with critical metrics for other actions with respect to adapted variations, for example, \eqref{Jwidetildeex}.
Since the case of $p=1$ is trivial for this action, we consider $n=1$ instead.
Next result provides applications to foliations whose leaves have constant second mean curvature, see \cite[Section~1.1.1]{rw-m}.

\begin{prop}
Let
$\widetilde{\cal D}$ be spanned by a unit vector field $N$ on a complete pseudo-Riemannian manifold~$(M,g)$.
If $g$ is critical for the action \eqref{Jwidetildeex} with respect to adapted variations then
$\widetilde{\Sm}_{\,\rm ex}(\Omega, g)\le0$ and
\begin{equation}\label{tau2extrinsic}
 \tilde\tau_{1} = 0,\quad \tilde\tau_{2} = \const.
\end{equation}
\end{prop}

\begin{proof}
From \eqref{CorEwidetildeex-main-0i} we obtain
\[
 \nabla_{N}\tilde h_{sc} -\tilde\tau_{1}\tilde h_{sc}
 +\epsilon_{N} [\,{\tilde T}^{\sharp}_{N},{\tilde A}_{N}]^\flat = 0.
\]
Tracing the above yields $N( \tilde\tau_{1} ) = \tilde\tau_{1}^{2}$, and in view of completeness of the metric,
the only solution is $\tilde\tau_{1}=0$, hence \eqref{tau2extrinsic}$_1$.
From \eqref{CorEwidetildeex-main-0i-dual}  with $n=1$
and $\Phi_{\tilde h}=\epsilon_{N}(\tilde\tau_1^2-\tilde\tau_2)\,\tilde g$
we obtain
\[
 \tilde\tau_{1}^{2} - \tilde\tau_{2}  = \epsilon_{N}\widetilde{\Sm}_{\,\rm ex},\quad
 \widetilde{\Sm}_{\,\rm ex}=\widetilde{\Sm}_{\,\rm ex}(\Omega, g),
\]
which together with $\widetilde{\Sm}_{\,\rm ex} = \widetilde{\Sm}\,_{\rm ex}^{*}$
and $\tilde\tau_{1}=0$ yields $\tilde\tau_{2} = -\epsilon_{N} \widetilde{\Sm}\,_{\rm ex}^{*}(\Omega, g)$.
Hence critical metrics of \eqref{Jwidetildeex} with respect to adapted variations
are those with constant $\tilde\tau_{2}$.
\end{proof}

For $n=1$ the critical metrics of the action \eqref{Jwidetildeex} with respect to adapted variations
also satisfy the differential equation
\[
 \nabla_{N}\,\tilde h_{sc} + \epsilon_{N}[\,{\tilde T}^{\sharp}_{N}, {\tilde A}_{N}]^\flat = 0,
\]
which in the case of integrable ${\cal D}$, together with \eqref{tau2extrinsic} yield the system of equations
studied in Section~\ref{subsec:codim1fol}, see \eqref{nablahsc0}
with interchanged $\cal D$ and $\cal \widetilde{D}$, and ${\tilde\tau_{2}}(\Omega, g)$ in place of $-\Ric_{N}(\Omega,g)$.

\subsection{Codimension-one foliations}
\label{subsec:codim1fol}

The structure theory and dynamics of codimension-one foliations on manifolds are fairly well understood.
The simplest examples of codimension-one foliations are the level surfaces of a function $u:M \to\RR$ with no critical points.
Geometric properties of such foliations correspond to analytic properties of their defining functions.
As a particular example one can consider isoparametric functions.
In the section we analyze
adapted critical metrics of the action \eqref{E-Jmix} for codimension-one foliations.

Let $\calf$ be a codimension-one foliation
with a~normal $N\in{\mathfrak X}_M$ of a pseudo-Riemannian manifold $(M^{n+1},g)$.
Assume that $|g(N,N)|=1$ and denote $\epsilon_{N} = g(N,N)$.
 We have, see \eqref{E-Rictop2},
\[
 {r}_{\,\cal D} =\epsilon_{N}\Ric_{N}\,g^\perp,\quad
 {r}_{\,\calf} =\epsilon_{N}(R_N)^\flat,
\]
where $R_N\!=R(N,\,\cdot\,)N$ is the {Jacobi operator} and $\Ric_{N}=\sum_a \eps_a\,g(R_N(E_a),E_a)$.
Then again the action \eqref{E-Jmix} reduces itself to \eqref{E-Jmix-N}.
 Let $h_{sc}$ be the scalar second fundamental form, and $A_N$ the Weingarten operator of ${\calf}$.
We~have $T=0=\tilde T$ and
\[
 h_{sc}(X,Y)=\epsilon_{N}\,g(\nabla_X\,Y,\,N),\quad A_N(X)=-\nabla_X\,N,\quad
 (X,Y\in T\calf).
\]
Define the functions $\tau_i = \tr A_N^i\ (i\ge0)$, see \cite{rw-m},
which can be expressed using the elementary symmetric functions $\sigma$'s,
\[
 \det(\id+\,t\,A_N)=\sum\nolimits_{\,k\le n}\sigma_{k}\,t^{k},
\]
called \textit{mean curvatures}.
For example, $\tau_1 =\epsilon_{N}\tr h_{sc}$ is the mean curvature of~${\calf}$ and
\[
 \tau_1=\sigma_1=\tr A_N=-\Div N,\quad
 \tau_2=\sigma_1^2-2\,\sigma_2=\tr A_N^2.
\]
Notice that
${\cal A}=\epsilon_{N}A_N^2$ and $\widetilde{\cal A} =g(\tilde H,\tilde H) N$, where
 $\tilde H=\epsilon_{N}\nabla_{N}\,N$
is the curvature vector of $N$-curves.
By \eqref{E-genRicN-2}$_1$
and $\widetilde\Psi=\tilde H^\flat\otimes\tilde H^\flat$,
we obtain
\begin{equation}\label{E-genRicN-p1}
 \epsilon_{N}(R_N + A_N^2)^\flat = \nabla_N\,h_{sc} -\tilde H^\flat\otimes \tilde H^\flat
 +\epsilon_{N}{\rm Def}_{\calf}(\tilde H).
\end{equation}
Then we find, taking trace of \eqref{E-genRicN-p1}, that (see also \cite{rw-m,wa1})
\begin{equation}\label{eq-ran1}
 \Ric_{N} = N(\tau_1)-\tau_2 +\Div \tilde H.
\end{equation}
 It is easy to see that \eqref{E-S*main} takes the form
\begin{equation}\label{E-S*-1}
 \Sm^*_{\rm mix} = \epsilon_{N}\Ric_{N}
 -2\,\epsilon_{N}\left\{\begin{array}{cc}
  N(\tau_1)-\tau_2 & ({\rm for}~g^\perp\mbox{\rm-variations}),\\
  \frac1n\Div\tilde H & ({\rm for}~\tilde g\mbox{\rm-variations})\,.
\end{array}\right.
\end{equation}

From Theorem~\ref{T-main00} (or Corollary~\ref{T-main01}) we obtain the following.

\begin{prop}[\bf Euler-Lagrange equations]\label{T-main02}
Let $\calf$ be a codimension-one foliation of a pseudo-Riemannian manifold $(M^{n+1},g)$,
whose normal distribution ${\cal D}$ is spanned by a unit vector field $N$.
If $g$ is critical for the action \eqref{E-Jmix-N} with respect to adapted variations then
\begin{eqnarray}\label{E-main-i2}
 &&\hskip-13mm
 \Ric_{N} + \,\epsilon_{N}\Sm^*_{\rm mix}(\Omega,g) -(N(\tau_1)-\tau_1^2) -\Div\tilde H = 0
 \quad ({\rm for}~g^\perp\mbox{\rm-variations}),\\
\label{E-main-ii1}
\nonumber
 &&\hskip-13mm
 \epsilon_{N}(R_N +A_N^2)^\flat-\tau_1 h_{sc} +\tilde H^\flat\otimes\tilde H^\flat
 -\epsilon_{N}\,{\rm Def}_{\calf}(\tilde H) \\
 && =\frac12\,\big(\epsilon_{N}\Ric_{N} - \Sm^*_{\rm mix}(\Omega,g) +\epsilon_{N}\Div(\tau_1 N-\tilde H) \big)\,\tilde g
 \quad ({\rm for}~\tilde g\mbox{\rm-variations}).
\end{eqnarray}
One may rewrite \eqref{E-main-i2}\,--\,\eqref{E-main-ii1} equivalently,
using \eqref{E-genRicN-p1}\,--\,\eqref{eq-ran1}, as
\begin{eqnarray}\label{E-RicNs0F}
 &&\tau_1^{2}-\tau_2 = -\,\epsilon_{N}\Sm^*_{\rm mix}(\Omega,g) \quad ({\rm for}~g^\perp\mbox{\rm-variations}),\\
\label{E-RicNs1F}
 &&\hskip-13mm \nabla_N h_{sc} {-}\tau_1 h_{sc} = \frac12\,\big(
 2\,\epsilon_{N}(N(\tau_1)-\tau_1^2)
 +\epsilon_{N}(\tau_1^{2}-\tau_2) -\Sm^*_{\rm mix}(\Omega,g)\big)\,\tilde g
 \ \ ({\rm for}~\tilde g\mbox{\rm-variations}).
\end{eqnarray}
\end{prop}


\begin{rem}\rm
Note that adapted variations provide the same Euler-Lagrange equations as in \cite{bdrs},
where the action \eqref{E-Jmix-N} was examined in a foliated globally hyperbolic space-time,
the  Euler-Lagrange equa\-tions (called the mixed gravitational field equations) were derived using variation formulas for the curvature tensor, then their linea\-rization and solution for an empty space have been obtained.
There, ${\cal D}$ was spanned by a unit (for initial metric $g$), time-like vector field $N$ with integrable orthogonal distribution~$\widetilde{\cal D}$.
Equations \eqref{E-RicNs0F} and \eqref{E-RicNs1F} are there formulated in terms of a newly introduced tensor $\Ric_{\cal D}(g)$, whose trace is denoted by ${\rm Scal}_{\cal D}(g)$.
For unit vectors $X,Y\in\widetilde{\cal D}$ we have
\begin{eqnarray*}
 && \Ric_{\cal D}(g)(X,Y) = \big(\nabla_{N}\,{h}_{sc} -{\tau}_{1}{h}_{sc}\big)(X,Y),\\
 && \Ric_{\cal D}(g) (X,N) = \Div({A}_N(X)),\quad
 \Ric_{\cal D}(g) (N,X) = - \Div({A}_N(X)),\\
 && \Ric_{\cal D}(g) (N,N) = - \Div H,
\end{eqnarray*}
and the Euler-Lagrange equations
for the action \eqref{E-Jmix-N} take the following form:
\begin{equation}\label{E-gravity}
 \Ric_{\cal D}(g) - \frac12\,{\rm Scal}_{\cal D}(g)\,g
 +\Ric_{N}\big(N^{\flat}\otimes N^{\flat} +\frac{1}{2}\,g\big) = 0\,.
\end{equation}
Since one should actually use only the symmetric part of $\Ric_{\cal D}(g)$ in \eqref{E-gravity}, its both sides vanish when evaluated on $(X,N)$. Also, \eqref{E-gravity} reduces to \eqref{E-RicNs0F} when evaluated on $(N,N)$ (with $\Sm^*_{\rm mix}(\Omega,g) =0$, because in \cite{bdrs} the volume preserving variations are not considered), while evaluating \eqref{E-gravity} on $X,Y\in\widetilde{\cal D}$ yields \eqref{E-RicNs1F}.
\end{rem}

\begin{lem}\label{L-trace-tau1}
Let $g$ be critical for \eqref{E-Jmix-N} with respect to adapted variations.
Then the function $\Div(\nabla_{N}\,N)$ is non-positive somewhere in $\Omega$,
and Euler-Lagrange equations \eqref{E-RicNs0F}\,--\,\eqref{E-RicNs1F} read
\begin{equation}\label{E-RicNs1F-2}
 \tau_1^{2}-\tau_2 = \Ric_{N}(\Omega,g) - 2\,\hat C,\qquad
 \nabla_N h_{sc} -\tau_1 h_{sc} = \frac{\epsilon_N}{n}\,\hat C\,\tilde g,
\end{equation}
where $\hat C\le0$ and $\tau_1$ is a global solution of the following ODE (on $N$-lines):
\begin{equation}\label{E-tau-const0}
 N(\tau_1)-\tau_1^2=\hat C.
\end{equation}
\end{lem}

\proof Denote~by
\[
 X=(N(\tau_1)-\tau_1^2)(\Omega,g),\ \ Y=(\tau_1^2-\tau_2)(\Omega,g),\ \
 Z=(\Div\tilde H)(\Omega,g),\ \ J=\Ric_{N}(\Omega,g).
\]
Integrating \eqref{E-RicNs0F} and using \eqref{E-S*-1}, we obtain $2X+Y=J$.
Integrating trace of \eqref{E-RicNs1F} and using \eqref{E-RicNs0F} and \eqref{E-S*-1}, we obtain $2(n-1)X+nY+2Z=nJ$.
The~rank 2 linear system
\[
 \{2X+Y=J,\quad 2(n-1)X+nY+2Z=nJ\}\quad ({\rm with\ variables}\ \  X,Y,Z)
\]
admits 1-parameter family of solutions $X=Z=\hat C,\ Y=J-2\,\hat C$, where $\hat C\in\RR$.
Note that the integral of identity \eqref{eq-ran1} is $X+Y+Z=J$, which is also satisfied by the above solution.
Hence, tracing \eqref{E-RicNs1F}, we obtain \eqref{E-tau-const0}.
If $\hat C\ge0$ then the only global solution of \eqref{E-tau-const0} is $\tau_1\equiv0$ (and hence, $\hat C=0$),
otherwise, i.e., $\hat C < 0$, any global solution $\tau_1(t)\ (t\in\RR)$ is bounded: $\tau_1^2(t)\le{|\hat C|}$
and given by
\begin{equation}\label{E-tau-N}
 \tau_1(t) = {|\hat C|}^{1/2}\,\Big(1-\frac{2({|\hat C|}^{1/2}-\tau_1^0)}
 {({|\hat C|}^{1/2}\!+\tau_1^0)\,e^{-2\,t\,{|\hat C|}^{1/2}}\!+{|\hat C|}^{1/2}\!-\tau_1^0}\Big),
  \ \ \tau_1(0)=\tau_1^0\in [-{|\hat C|}^{\frac12},{|\hat C|}^{\frac12}],
\end{equation}
including constant solutions $\tau_1\equiv\pm\,{|\hat C|}^{1/2}$.
Since $Z\le0$, the function $\Div(\nabla_{N}\,N)$ is non-positive somewhere in $\Omega$,
and \eqref{E-RicNs0F}\,--\,\eqref{E-RicNs1F} read as \eqref{E-RicNs1F-2}.
\qed

\smallskip

A~coordinate system $(x_0,x_1,\ldots,x_n)$
such that the leaves are given by $\{x_0=c\}$ and $N$-curves are $x_1=c_1,\ldots,x_n=c_n$
is called a \textit{biregular foliated chart}\index{biregular foliated coordinates}, see \cite[Section~5.1]{cc1}.
If~$(M, \calf, g)$ is a foliated pseudo-Riemannian manifold and $N$ is the unit normal, then
in bire\-gular foliated coordinates
the metric $g$ has the form
 $g=g_{00}\,dx_0^2+\sum\nolimits_{i,j>0}g_{ij}\,dx_i dx_j$.
Denote by
$g_{ij,k}$ the derivative of $g_{ij}$ in the $\partial_k$-direction.
For orthogonal biregular foliated coordinates we have
$g_{00}=\epsilon_N |g_{00}|$, $g_{ii}=\epsilon_i |g_{ii}|$ and $g_{ij}=0\ (i\ne j)$.

\begin{lem}\label{lem-biregG}
For a pseudo-Riemannian metric in orthogonal $($i.e., $g_{ij}=\delta_{ij}\,g_{ii})$
biregular foliated coordinates of a codimension-one foliation on $(M,g)$, one has
\begin{eqnarray*}
 N \eq \partial_0/\sqrt{|g_{00}|}\quad\mbox{\rm(the unit normal)},\\
 h_{ij}\eq
\Gamma^0_{ij}\sqrt{g_{00}}=
 -\frac12\,\epsilon_N\,\delta_{ij}\,g_{ii,0}/\sqrt{|g_{00}|}
 \quad\mbox{\rm(the second fundamental form)},\\
 A^j_{i}\eq
-\Gamma^j_{i0}/\sqrt{|g_{00}|}=
 -\frac{1}{2\,\sqrt{|g_{00}|}}\,\delta_{i}^{j}\,\frac{g_{ii,0}}{g_{ii}}
 \quad\mbox{\rm(the Weingarten operator)},\\
 \tau_1 \eq -\frac{1}{2\sqrt{|g_{00}|}}\sum\nolimits_{\,i>0}\,\frac{g_{i i,0}}{g_{i i}},\quad
 \tau_2=\frac{1}{4\,|g_{00}|}\sum\nolimits_{\,i>0}\,\Big(\frac{g_{i i,0}}{g_{i i}}\Big)^2,\quad \mbox{\rm etc}.
\end{eqnarray*}
\end{lem}

\proof This is similar to the proof of Lemma~2.2 in \cite{rw-m} for Riemannian case.
\qed

\smallskip

Let $(x_0=t,x_1,\ldots x_n)$ be biregular orthogonal foliated coordinates on $M^{n+1}=\RR\times\RR^n$ with the foliation $\{x_0=c\}$,
components $g_{00}, g_{11},\ldots g_{nn}$, and $N = |g_{00}|^{-1/2}\,\partial_{t}$.
 Let $g\in{\rm Riem}(M,\,{T\calf},\,{\cal D})$ be a critical point of the action \eqref{E-Jmix-N}
with respect to adapted variations supported in $\Omega\subset M$.
 Then $\hat C\le0$, see Lemma~\ref{L-trace-tau1}, and $\tau_{1}$ is a bounded function: $\tau_{1}^2\le|\hat C|$, see \eqref{E-tau-N}.
Let $g_{00}\ne0$ be an arbitrary smooth function on $M$.
Using Lemma~\ref{lem-biregG}, we obtain
\begin{equation}\label{E-Qi}
  (\nabla_N\,h_{sc})_{ii} = -\frac{\epsilon_N}{2\,|g_{00}|}\,\big( g_{ii,00}
 - \frac12\,g_{ii,0}(\log|g_{00}|)_{,\,0} - (g_{ii,0})^2/g_{ii} \big).
\end{equation}
By~\eqref{E-Qi}, Euler-Lagrange equation \eqref{E-RicNs1F-2}$_{2}$ becomes the system of $n$ independent equations:
\begin{equation}\label{giitauconst}
 g_{ii,\,00} -\frac{1}{g_{ii}}\,(g_{ii,\,0})^{2}
 - g_{ii,\,0}\big(\,\frac{1}{2}\,(\log|g_{00}|)_{,\,0}
 +\tau_{1}\sqrt{|g_{00}|}\,\big)
 +\frac2{n}\,\hat C\,|g_{00}|\,g_{ii}
 = 0
\end{equation}
for $i=1,\ldots n$.
We seek solutions of \eqref{giitauconst} in the following form:
\begin{equation}\label{E-g1122}
 g_{ii}=\epsilon_i\,f_i(x_1,\ldots x_n)\,e^{\,-2\int \sqrt{|g_{00}|}\,y_i(t)\,{\rm d}\,t},\quad
\end{equation}
where $f_i\ (i=1,\ldots n)$ are arbitrary positive functions. From Lemma~\ref{lem-biregG} it follows that
for the metric given by \eqref{E-g1122}
the Weingarten operator has diagonal view and the functions $y_{1},\ldots,y_{n}$ are its eigenvalues, i.e., principal curvatures.
Hence, they must satisfy
\begin{equation}\label{E-y1y2tau}
 y_1(t)+\ldots + y_n(t) = \tau_{1},\quad y_{1}^{2} +\ldots + y_{n}^{2} = \tau_{2}.
\end{equation}
The metric given by \eqref{E-g1122} is critical with respect to adapted variations if and only if
\eqref{E-RicNs1F-2}$_1$ holds
and all $y_i(t)$  solve the first-order linear ODE
\begin{equation}\label{odetauconst}
  y'(t) -\tau_{1}\sqrt{|g_{00}|}\,y(t) -\frac{1}{n}\,\hat C\,\sqrt{|g_{00}|} = 0,
\end{equation}
where $\tau_1$ is given by \eqref{E-tau-N} with $\tau_1^0=f_0(x_1,\ldots,x_n)$.
Note that
\eqref{odetauconst} corresponds to \eqref{E-RicNs1F-2}$_2$.

For any $n>2$, the only critical metrics of the form \eqref{E-g1122} and $\tau_1=0$
are ones with constant principal curvatures $y_{i}$, $i \in \{1, \ldots, n\}$, see \eqref{odetauconst}
and case 2 of Example~\ref{Ex-tau1} in what follows.
One may use arbitrary $y_{i}$ satisfying equations \eqref{E-y1y2tau}:
$y_{1} + \ldots + y_{n} =0$ and
$y_{1}^{2} + \ldots + y_{n}^{2} = -\Ric_{N}(\Omega,g)$.
It follows that $\Ric_{N}(\Omega,g) \le0$, and if $\Ric_{N}(\Omega,g) =0$, the only solution is
$y_i=0$ -- a~totally geodesic foliation.
Again (for any $n>2$ and $\tau_1=0$) we can consider metrics with constant $\Div\nabla_NN$,
see Example~\ref{Ex-tau1}, case 2.
For a function of the view $|g_{00}|=P(t)$, we have $\Ric_{N}=\const\le0$, and such foliation is isoparametric.


The next example analyzes the set of solutions to \eqref{E-RicNs0F}--\eqref{E-RicNs1F} for $n=2$.

\begin{example}\label{Ex-tau1}\rm
For $n=2$, from \eqref{E-y1y2tau}$_1$ and \eqref{E-RicNs1F-2}$_1$ we obtain a quadratic equation for $y$, from which it follows that
\begin{equation}\label{y1quad}
 y_{1,2} = \frac{1}{2}\,\tau_{1} \pm \frac{1}{2}\,\big(\tau_{1}^{2} - 4\,|\hat C| - 2\Ric_{N}(\Omega,g)\big)^{1/2}\,.
\end{equation}
Substituting \eqref{y1quad} into \eqref{odetauconst} yields two equations relating $\tau_{1}$ with $g_{00}$,
\begin{equation}\label{sqrtg00}
 \sqrt{ |g_{00}| } = \mp\frac{\dt\tau_{1} \cdot (\tau_{1} \pm \sqrt{\tau_{1}^{2} + G})}
 {(\,|\hat C| -\tau_{1}^{2} \mp \tau_{1} \sqrt{\tau_{1}^{2} + G})\,\sqrt{\tau_{1}^{2}+G}}\,,
\end{equation}
where $G = -4\,|\hat C| - 2\,\Ric_{N}(\Omega,g)$,
and we already know that $\tau_{1}$ satisfies \eqref{E-tau-N}.
For $\dt\tau_{1} \ne 0$ we obtain 2 different values for $|g_{00}|$ -- a contradiction.
For $\dt\tau_{1}\equiv0$, \eqref{sqrtg00} yields a contradiction: $g_{00} =0$,
unless $\tau_{1}^{2} + G =0$ or $- \tau_{1}^{2} \mp \tau_{1} \sqrt{\tau_{1}^{2} + G} + | \hat C | =0$.
We shall see that this is exactly what happens for $n=2$, when we treat cases of $\tau_{1}=\const$ separately,
due to Lemma~\ref{L-trace-tau1}.

\smallskip
1. Let $\tau_1=\pm\,{|\hat C|}^{\frac12}$ with $\hat C<0$.
Then from \eqref{E-RicNs1F-2}$_{1}$ we obtain $\tau_2 = -|\hat C|- \Ric_{N}(\Omega,g)\ge0$.
The~principal curvatures of the leaves obey $y_1y_2=|\hat C|+\frac12\,\Ric_{N}(\Omega,g)$ and are constant:
\begin{equation}\label{E-k12}
 y_{1,2}=\frac12\,|\hat C|^{\frac12}\pm\frac12\,\big(-2\,\Ric_{N}(\Omega,g)-3|\hat C|\big)^{\frac12}.
\end{equation}
Moreover, $\Ric_{N}(\Omega,g)\le-\frac32\,|\hat C|$ holds.
For $g_{11}$ and $g_{22}$ represented by \eqref{E-g1122}, we~get
\begin{equation}\label{E-g1122-b}
 g_{11}=\epsilon_1 f_1(x_1,x_2)\,e^{\,-2\,y_1\int\sqrt{|g_{00}|}\,{\rm d}\,t},\quad
 g_{22}=\epsilon_2 f_2(x_1,x_2)\,e^{\,-2\,y_2\int\sqrt{|g_{00}|}\,{\rm d}\,t}.
\end{equation}
The metric \eqref{E-g1122-b} is critical for $J_{\rm mix}$ with respect to adapted variations
if and only if \eqref{E-RicNs1F-2}$_{2}$ holds, i.e., $y_{1}$ and $y_{2}$ are both solutions of \eqref{odetauconst}.
The only constant solution of \eqref{odetauconst} is $y=\frac1n\,|\hat C|^{1/2}$.
Comparing this result to \eqref{E-k12}, we see that there exists a metric of the form \eqref{E-g1122-b}
critical with respect to adapted variations if and only if
\begin{equation}\label{E-J-eq32}
  \Ric_{N}(\Omega,g) = -\frac32\,|\hat C|.
\end{equation}
 In this case, we have $y_{1,2}=\frac12\,|\hat C|^{\frac12}$, and from \eqref{E-g1122-b} obtain
\begin{equation*}
 g_{11}=\epsilon_1 f_1(x_1,x_2)\,e^{\,-|\hat C|^{1/2} \int\sqrt{|g_{00}|}\,{\rm d}\,t},\quad
 g_{22}=\epsilon_2 f_2(x_1,x_2)\,e^{\,-|\hat C|^{1/2} \int\sqrt{|g_{00}|}\,{\rm d}\,t}.
\end{equation*}
In our case, we cannot assume $|g_{00}|\equiv1$ (a~Riemannian foliation), since this yields $\tilde H=0$;
hence, a contradiction: $\hat C=0$.  Note also that from \eqref{E-J-eq32} it follows that
\[
 \tau_{1}^{2} + G = |\hat C| - 4\,|\hat C| -2\,\Ric_{N}(\Omega,g) = 0,
\]
thus, we cannot use~\eqref{sqrtg00}.
Using Lemma~\ref{lem-biregG}, we obtain:
\begin{equation} \label{divNN}
 \Div\tilde H = \sum\nolimits_{\,i>0}\big(\epsilon_i g_{ii} Q_{i,i} +\frac12\big(\epsilon_N g_{00,i}
 +\sum\nolimits_{\,j>0} \epsilon_j\,g_{jj,i} \big)\,Q_i\big),\ {\rm where}\ Q_i=-\frac1{2|g_{00}|}\,\frac{g_{00,i}}{g_{ii}}.
\end{equation}
Next, we will find condition for $\Div\tilde H$ to be
constant:
\begin{equation}\label{E-Z0}
 \Div\tilde H=Z_0=\const.
\end{equation}
Note that in our case, $Z_0\ne0$,
see Lemma~\ref{L-trace-tau1}.
Then by \eqref{eq-ran1}, we will get $\Ric_{N}=\const$; thus,
$g$ will be critical for any domain~$\Omega\subset M$.
To solve \eqref{E-Z0}, assume for simplicity that
\[
 f_a=1,\quad \epsilon_a=1\ \ (a=1,2),\quad
 \epsilon_N=1,\quad
 g_{00}=w(x_1,x_2) T(t)
\]
for some functions $w>0$ and $T>0$.
Then $g_{11}=g_{22}$ are functions of $t$ and $w$.
Hence, equation \eqref{E-Z0} yields an elliptic PDE (with parameter $t$) for $w$:
\begin{eqnarray*}
 && \Delta w +f(t,w)\,\<\nabla w,\nabla w\> = Z_0,\quad {\rm where}\\
 && f(t,w)= \frac12\,T (t)\, {e}^{\sqrt{|\hat C|\,w}\int\!\sqrt{T(t)}\,{\rm d}t}
 +\frac{\sqrt{C}\int\!\sqrt{T(t)}\,{\rm d}t} {2\,{w^{1/2}}}\Big(1+\frac{1}{2\,w\,T(t)}\Big)
 -\frac1{w}\,.
\end{eqnarray*}
The substitution $u=\int\frac{{\rm d}w}{F(w)}$ with $F(w)=e^{\,\int f(w)\,{\rm d}w}$ leads to the
Poisson's equation
 $\Delta u = Z_0$.
\newline
Hence, $u=u_0(x_1,x_2)+\frac12 Z_0(x_1^2+x_2^2)$, where $u_0$ is a harmonic function.

\smallskip

2. For $\tau_{1}=0$ (and $\hat C=0$, a harmonic foliation), the system \eqref{E-RicNs1F-2} reads:
\begin{equation}\label{nablahsc0}
 \tau_{2} = -\Ric_{N}(\Omega,g),\qquad \nabla_{N}\,h_{sc} =0
\end{equation}
with $\Ric_{N}(\Omega,g)\le0$.
In our case, the system \eqref{giitauconst} has the following view:
\begin{equation*}
 g_{ii,\,00} -\frac{1}{g_{ii}}\,(g_{ii,\,0})^{2} - \frac{1}{2}\,g_{ii,\,0}(\log|g_{00}|)_{,\,0} = 0,\quad i=1,2,
\end{equation*}
and $y'=0$, see \eqref{odetauconst}.
Thus, \eqref{E-g1122} are valid, where $y_1=-y_2$ are constant.
In view of \eqref{nablahsc0}$_1$ and assumption $\tau_{1}=0$, the~principal curvatures of the leaves~are
\[
 y_{1,2}=\pm\,(-\Ric_{N}(\Omega,g)/\,2)^{1/2}.
\]
Observe that we cannot use equation \eqref{sqrtg00}, because $|\hat C|-\tau_{1}^{2}\mp\tau_{1}\sqrt{\tau_{1}^{2}+G}=0$.

We can consider metrics with constant $\Div\nabla_NN$, see \eqref{E-Z0}, which will be critical for variations with respect to any $\Omega$.
For such metrics it follows from the assumption $\tau_{1}=0$ and Lemma~\ref{L-trace-tau1} that $Z_0=0$.
For a function of the view $|g_{00}|=P(t)$, by \eqref{divNN} we have $\Div\tilde H=0$ (since $Q_i=0$),
hence, we obtain a Riemannian harmonic foliation with $\Ric_{N}=\const\le0$.
Such foliation is given by an \textit{isoparametric function}~$x_{0}$.
\end{example}

\begin{defn}[see Chap.~8 in \cite{to}]\rm
A smooth function $f : M \rightarrow \mathbb{R}$ without critical points on a pseudo-Riemannian manifold $(M,g)$ is called \textit{isoparametric} if for any vector $X$ tangent to a level hypersurface of $f$ the following conditions are satisfied:
\begin{equation}\label{E-isopar2}
 X (g( \nabla f, \nabla f )) = 0,\qquad X(\Delta f) = 0.
\end{equation}
\end{defn}

\begin{prop}[see \cite{to}]
Let $\mathcal{F}$ be a foliation of a pseudo-Riemannian manifold $(M,g)$ by the level hypersurfaces of a function $f$
without critical points on $M$. Then the following conditions are equivalent:


 $(i)$ $\mathcal{F}$ is a Riemannian foliation and every its leaf has constant mean curvature;

 $(ii)$ $f$ is an isoparametric function.

\noindent
For Riemannian foliations of space forms, $(ii)$ is equivalent to the constancy of all principal curvatures on each level hypersurface of $f$.
\end{prop}

In biregular foliated coordinates, consider a function $f= x_{0}$. Then we have $g( \nabla f , \partial_{i} ) =0$ for $i>0$ and $g( \nabla f, \partial_{0}) = 1$; hence,
\[
 \nabla f = \epsilon_{N}\,g(\nabla f, N)N = \epsilon_{N} \frac{1}{| g_{00} |}\,\partial_{0} = \epsilon_{N}\frac{1}{\sqrt{| g_{00} |}}\,N,\qquad
 g(\nabla f, \nabla f) = \frac{g_{00}}{|g_{00}|^{2}} = \frac{\epsilon_{N}}{|g_{00}|}.
\]
Remark that a foliation by level hypersurfaces of $f$ is \textit{Riemannian} if and only if $X (| g_{00} |) =0$ for all $X$ tangent to the leaves, see \cite{to}, which is equivalent to \eqref{E-isopar2}$_{1}$. On the other hand, for such foliations, using
the definition $\Delta f = \epsilon_N g(\nabla_N\nabla f, N)+\sum_a \epsilon_a g(\nabla_a\nabla f, E_a)$, we have
\begin{eqnarray*}
 X(\Delta f) \eq X(\sum\nolimits_{\,a} \epsilon_{a}\,g(\nabla_{a}\,\epsilon_{N}\frac{1}{\sqrt{| g_{00} |}}\,N, E_{a}))
  + \epsilon_{N} X(g(\nabla_{N}\,\epsilon_{N}\frac{1}{\sqrt{| g_{00} |}}\,N, N))  \\
 \eq -\frac{\epsilon_{N}}{\sqrt{| g_{00} |}}\,X(\tau_{1}) + X\Big( N\big(\frac{1}{\sqrt{| g_{00} |}} \big) \frac{1}{|g_{00}|} \Big)  = -\frac{\epsilon_{N}}{\sqrt{| g_{00} |}}\,X(\tau_{1});
\end{eqnarray*}
hence, the mean curvature is constant along the leaves if and only if \eqref{E-isopar2}$_{2}$ holds.


\subsection{Conformal submersions}

Conformal submersions form an important class of mappings, which were investigated also in relation with Einstein equations,
see survey in \cite{Falcitelli}.

\begin{defn}\rm
Let $(M^{n+p}, g)$, $(\hat{M}^p, \hat{g} )$ be smooth pseudo-Riemannian manifolds. A differentiable mapping $\pi : (M, g) \rightarrow (\hat{M}, \hat{g} )$ is called a \emph{conformal} (or: \emph{horizontally conformal}) \emph{submersion} if
\begin{enumerate}
\item $\pi$ is a submersion, i.e., it is surjective and has maximal rank,
\item ${\rm d}\pi$ restricted to the distribution orthogonal to the fibers of $\pi$ is a conformal mapping.
\end{enumerate}
\end{defn}

Note that for $p=1$ any submersion is conformal.
Using our notation, we can define a conformal submersion as a mapping $\pi : (M, g) \rightarrow (\hat{M}, \hat{g} )$ for which $\cal \widetilde{D}$ is the maximal distribution such that ${\rm d}\pi ( {\cal \widetilde{D}} ) = 0$ and
\begin{equation}\label{dilation}
 \pi^{*}(\hat{g}) = e^{-2f}\,g^{\perp}
\end{equation}
for some $f \in C^{\infty}(M)$ (such $f$ is called the \emph{dilation} of the submersion). Then $\cal \widetilde{D}$ is tangent to the fibers of the submersion (and hence, integrable).
Denote $\nabla^\top f = (\nabla f)^\top$ and $f$ is as in \eqref{dilation}.
One can also show \cite{Gudmundsson} that $\cal D$ is totally umbilical with the second fundamental form satisfying:
\begin{equation}\label{hconfsub}
 \tilde h = - (\nabla^\top f)\, g^{\perp}\,.
\end{equation}

Among conformal submersions, those with totally umbilical fibers are example of particularly interesting geometry.
While the adapted variations \eqref{E-Sdtg}--\,\eqref{E-Sdtg-2} preserve the orthogonality of two distributions,
we can consider their particular class which preserves the structure of conformal submersion with totally umbilical fibers.

\begin{defn}\rm
We say that a variation $g_{t}$ is \emph{$\cal D$-conformal} if $\dt g^{\perp}_t = s g^{\perp}_0$ for some $s \in C^{\infty }(M)$, we define \emph{$\cal\widetilde{D}$-conformal} variations analogously and say that variation is \emph{biconformal} if it is both $\cal D$-conformal and $\cal \widetilde{D}$-conformal.
\end{defn}

A~tensor ${B}\in{\mathfrak M}_{\cal D}$ is ${\cal D}$-\emph{conformal}
if ${B}=s\,g^\perp$ for some $s \in C^\infty (M, \RR )$.
Given $g \in {\rm Riem}(M,\widetilde{\cal D},\, {\cal D})$, the subspace
of ${\mathfrak M}$, consisting of \textit{biconformal} adapted tensors, splits into the direct sum
of ${\cal D}$- and $\widetilde{\cal D}$-conformal components.

\begin{prop}
Let $\pi : (M,g) \rightarrow ( {\hat M} , {\hat g} )$ be a conformal submersion with totally umbilical fibers, and let $g_{t}$ be an adapted variation of $g$. Then all mappings $\pi : (M,g_{t}) \rightarrow ({\hat M},  {\hat g})$ are conformal submersions with totally umbilical fibers if and only if variation $g_{t}$ is ${\cal D}$-conformal~and
\begin{equation}\label{umbpreserving}
 \nabla \big( B - \frac{1}{n}\,(\,\tr B^{\sharp} )\,\tilde g\,\big) =0 .
\end{equation}
\end{prop}

\begin{proof}
If all the mappings $\pi : (M,g_{t}) \rightarrow ({\hat M}, {\hat g})$ are conformal submersions, we have
\[
  e^{\,-2f_t}\,g^{\perp}_t = \pi^{*}(\hat{g})
\]
for some $f_t \in C^{\infty}(M)$. Differentiating the above we obtain
\[
 e^{-2f_t}\dt g^{\perp}_t - 2\,\dt f_{t}\,e^{-2f_{t}} g_{t}^{\perp} =0.
\]
Hence, $\dt g^{\perp}_t = s g^{\perp}_0$ for $s = 2\,\dt f_{t}$.

If $\cal\widetilde D$ is totally umbilical for all $g_t$ then $h = \frac{1}{n} H \tilde g_t$,
and from \eqref{eq-hatbH-1} we obtain
\[
 \frac{2}{n} ( B(X,Y)\, H + g(X,Y)\, \dt H ) = \frac{2}{n}\, B(X,Y)\,H - \nabla B(X,Y)
\]
for all $X,Y \in \cal \widetilde{D}$.
Using \eqref{eq-hatH}$_{1}$ yields
\[
 \frac{1}{n}\,g(X,Y)\,\nabla(\tr B^{\sharp}) = \nabla B(X,Y).
\]
On the other hand, if \eqref{umbpreserving} is satisfied and the variation is $\cal D$-conformal, then from the uniqueness of the solution of ODE it follows that $h = \frac{1}{n} H \tilde g_t$ and $e^{-2f_t} g^{\perp}_t = \pi^{*} \hat{g}$ for all $g_t$;
hence, all $\pi : (M,g_{t}) \rightarrow ({\hat M},  {\hat g})$ are conformal submersions with totally umbilical fibers.
\end{proof}

Note that the condition \eqref{umbpreserving} is satisfied, in particular, by biconformal variations.

We examine the metrics critical for the action \eqref{E-Jmix} with respect to ${\cal D}$-conformal variations.
The Euler-Lagrange equation for these metrics is a scalar equation. To find it, we can use our result \eqref{E-Jmix-dt-fin},
with $B=s\,g^{\perp}$; by demanding it to be satisfied for all $s \in C^{\infty}(M)$ we obtain
\begin{equation} \label{biconfEL1}
 (p-1)\Div\tilde{H} + \frac{p-2}{2}\,( \Sm_{\rm ex} + \< \tilde T, \tilde T \> )
 +\frac{p}{2}\,\big(\widetilde\Sm_{\rm ex} + \< T, T \> - \Sm^*_{\rm mix}(\Omega,g) \big) = 0,
\end{equation}
where $\Sm^*_{\rm mix} = \Sm_{\rm mix} - \frac{2}{p}\,\big( \Sm_{\rm ex} + 2\,\<\tilde T, \tilde T \> - \<T, T\> \big)$.

The mixed scalar curvature is an important tool in investigation of conformal submersions with totally umbilical fibers.
In~\cite{Zawadzki}, it was used to obtain some integral formulas and existence conditions for such mappings.
There, the following formula was considered:
\begin{equation} \label{SmixCSUF}
 \Sm_{\rm mix} = - p\,\widetilde{\Delta} f - p\,g( \nabla^\top f, \nabla^\top f )
 + \< \tilde T, \tilde T \> + \Div H + \frac{n-1}{n}\,g(H,H),
\end{equation}
which is just a special case of \eqref{eq-ran-ex}, expressed in terms of $f$ and $H$. We can present in a similar way the Euler-Lagrange equations for biconformal variations on the domains of conformal submersions with totally umbilical fibers.

\begin{prop}[\bf Euler-Lagrange equations]
Let $\pi : (M^{n+p},g) \rightarrow ({\hat M}^p, {\hat g})$, where $p>1$, be a conformal submersion with totally umbilical fibers, and $g$ be critical for the action \eqref{E-Jmix} with respect to biconformal variations. Then
\begin{eqnarray}\label{ELbiconfCSUF1}
\nonumber
 &&\hskip-10mm -\,2\,p\,(p-1)\,\widetilde{\Delta} f - p^{2}(p-1)\,g(\nabla^{\top} f, \nabla^{\top} f)
 +\frac{(p-2)(n-1)}{n}\,g(H,H) \\
 && +\,(p-2)\,\< \tilde T, \tilde T \>
 = {p}\,\Sm\,^*_{\rm mix}(\Omega,g) \quad ({\rm for}\ {\cal D}\mbox{\rm-conformal\ variations}), \\
\label{ELbiconfCSUF2}
\nonumber
 &&\hskip-10mm p\,(p-1)(n-2)\,g( \nabla^{\top} f, \nabla^{\top} f) +2\,(n-1)\,\Div H +(n-1)\,g( H , H) \\
 && +\,{n}\,\< \tilde T, \tilde T \>
 = n\,\widetilde{\Sm}\,^*_{\rm mix}(\Omega,g) \qquad ({\rm for}\ \widetilde{\cal D}\mbox{\rm-conformal\ variations}),
\end{eqnarray}
where
\begin{eqnarray}\label{E-star2}
\nonumber
 &&\hskip-5mm \Sm\,^*_{\rm mix} = - p\,\big(\widetilde{\Delta} f + g( \nabla^{\top} f, \nabla^{\top} f)\big)
 + \frac{p-4}{p}\,\<\tilde T, \tilde T \> + \frac{(n-1)(p-2)}{np}\,g(H,H) + \frac{p-2}{p}\,\Div H , \\
 &&\hskip-5mm \widetilde{\Sm}\,^*_{\rm mix} = -p\,\frac{n-2}{n}\,\big(\widetilde{\Delta} f
   +g( \nabla^{\top} f, \nabla^{\top} f )\big) + \frac{n+2}{n}\,\< \tilde T, \tilde T \> + \Div H + \frac{n-1}{n}\,g(H,H).
\end{eqnarray}
\end{prop}

\begin{proof}
For conformal submersions with totally umbilical fibers we have
\[
 T=0,\quad \Sm_{\rm ex}=\frac{n-1}{n}\,g(H,H),\quad
 \Sm_{\rm ex}=\frac{p-1}{p}\,g(\tilde H,\tilde H),
\]
and from \eqref{hconfsub} we obtain $\tilde H=-p \nabla^{\top} f$. Using this, we rewrite \eqref{biconfEL1} as~\eqref{ELbiconfCSUF1}.
For $\widetilde{\cal D}$-conformal variations of metrics on the domain of conformal submersion with umbilical fibers, a formula analogous to \eqref{biconfEL1} yields \eqref{ELbiconfCSUF2}. Using \eqref{SmixCSUF} in \eqref{E-S-star}, we get remaining formulas
\eqref{E-star2}.
\end{proof}

We examine the above equations in a particular case of totally geodesic fibers, i.e., $H=0$.

\begin{prop}
Let $\pi : (M^{n+p},g) \rightarrow ({\hat M}^p, {\hat g})$, where $p>1$ and $g_{\,|\,\tilde{\cal D}}>0$,
be a conformal submersion with totally geodesic fibers. If $g$ is critical
for the action \eqref{E-Jmix} with respect to biconformal variations then
\[
 e^{\,\lambda f},\quad {\rm where}\quad \lambda = \frac1{2\,n}\,(p\,n + (p-2)(n-2))>0,
\]
is a fiberwise harmonic function.
\end{prop}

\begin{proof}
From \eqref{ELbiconfCSUF1} and \eqref{ELbiconfCSUF2} we obtain
\[
 p(p-1) \big( \widetilde{\Delta} f +
 \lambda\,g( \nabla^{\top} f, \nabla^{\top} f ) \big)
 =  \frac{p-2}{2}\,\widetilde{\Sm}\,^*_{\rm mix}(\Omega,g) - \frac{p}{2}\,\Sm^*_{\rm mix}(\Omega,g).
\]
Using the identity
 $ \widetilde{\Delta} f + \lambda\,g( \nabla^{\top} f, \nabla^{\top} f )
 =\frac{1}{\lambda}\,e^{-\lambda f} \widetilde{\Delta}\,e^{\lambda f}$ in the above yields
\begin{equation} \label{CSGF}
 \widetilde{\Delta}\,e^{\,\lambda f}
 = G\,\lambda\,e^{\,\lambda f},
\end{equation}
where
$G=\frac1{p\,(p-1)}\,\big( \frac{p-2}{2}\,\widetilde{\Sm}\,^*_{\rm mix} -\frac{p}{2}\,\Sm\,^*_{\rm mix} \big)(\Omega,g)$.
Equation \eqref{CSGF} is
an eigenvalue problem (with positive solution) on every fiber; hence, $G=0$ and $e^{\lambda f}$ is fiber-wise harmonic.
For closed fibers, \eqref{CSGF} admits only fiber-wise constant solutions $f$.
If we allow our variations not to preserve the volume of $\Omega$, then again $G=0$ and \eqref{CSGF} becomes the fiberwise Laplace equation for $e^{\lambda f}$.
\end{proof}

\begin{rem}\rm
The set of bounded (or positive) harmonic functions on open manifolds with nonnegative
curvature was described in \cite{LT}: in particular,
every unique -- with respect to multiplication by a
constant -- positive harmonic function corresponds to a 'large end' of the manifold (in our case, the fiber of the submersion).
\end{rem}



\end{document}